\documentclass[11pt]{amsart}
\usepackage{tikz-cd}
\usepackage{enumitem}
\usepackage{hyperref}
\setlist[description]{leftmargin=\parindent,labelindent=\parindent}
\usetikzlibrary{matrix,arrows,decorations.pathmorphing}
\usepackage{amssymb}
\usepackage{amsmath,amscd}
\usepackage{mathrsfs}
\usepackage{url}
\usepackage{cite}
\usepackage{fullpage}
\usepackage{hyperref}
\usepackage{verbatim}
\usepackage{comment}
\usepackage{fancyvrb}
\usepackage{fvextra}
\usepackage{caption}
\usepackage{ytableau}
\usepackage{stmaryrd}
\usepackage{appendix}
\allowdisplaybreaks

\newtheorem{thm}{Theorem}

\newtheorem{lem}[thm]{Lemma}
\newtheorem{cor}[thm]{Corollary}

\theoremstyle{definition}
\newtheorem{definition}[thm]{Definition}
\newtheorem{example}[thm]{Example}
\newtheorem{rem}[thm]{Remark}

\newcommand{\X}{\mathcal{X}}

\newcommand {\gp}{\text{gp}}

\newcommand{\qq}{\mathbb{Q}}
\newcommand{\R}{\mathsf{R}}

\newcommand{\M}{\mathcal{M}}

\newcommand{\A}{\mathcal{A}}

\newcommand{\CH}{\mathsf{CH}}

\DeclareMathOperator{\ct}{ct}

\renewcommand{\tilde}{\widetilde}

\DeclareMathOperator{\Sp}{Sp}

\DeclareMathOperator{\GL}{GL}
\DeclareMathOperator{\SL}{SL}

\DeclareMathOperator{\Spec}{Spec}

\newcommand{\Tor}{\mathsf{Tor}}

\DeclareMathOperator{\Sym}{Sym}

\DeclareMathOperator{\codim}{codim}
\DeclareMathOperator{\Hom}{Hom}

\newcommand{\NN}{\mathbb{N}}
\newcommand{\LG}{\mathsf{LG}}
\usepackage{wrapfig}

\renewcommand{\bar}{\overline}
\newcommand{\Mb}{\overline{\M}}

\title[Tautological projection for cycles on the moduli space of abelian varieties]{Tautological projection for cycles on the moduli space of abelian varieties}

\dedicatory{To Gerard van der Geer, for opening the door}

\author{Samir Canning}
\address{Department of Mathematics, ETH Z\"urich}
\email {samir.canning@math.ethz.ch}

\author{Sam Molcho}
\address{Department of Mathematics, Sapienza Universit\`a di Roma}
\email {samouil.molcho@uniroma1.it}

\author{Dragos Oprea}
\address{Department of Mathematics, University of California, San Diego}
\email {doprea@math.ucsd.edu}

\author{Rahul Pandharipande}
\address{Department of Mathematics, ETH Z\"urich}
\email {rahul@math.ethz.ch}

\date{}
\begin{document}
\baselineskip=17pt

\begin{abstract} 
We define a tautological projection operator for algebraic cycle classes on 
the moduli space of principally polarized abelian varieties $\A_g$: every cycle class decomposes
canonically as a sum of a tautological and a non-tautological part. 
The main new result required for the definition of the projection operator is 
the vanishing of the top Chern class of the Hodge bundle over the boundary $\overline{\A}_g\smallsetminus \A_g$ of any toroidal compactification $\overline{\A}_g$ of the moduli space $\A_g$. We prove the vanishing
by a careful study of residues in the boundary geometry.

The existence of the projection operator raises many natural questions
about cycles on $\A_g$. We calculate the projections of 
all product cycles $\A_{g_1}\times \ldots \times \A_{g_\ell}$ in terms
of Schur determinants, discuss
Faber's earlier calculations related to the Torelli locus, and
state several open questions. The Appendix contains a conjecture
about the projection of the locus of abelian varieties with real multiplication.

\end{abstract}

\date{May 2025}
\maketitle
\setcounter{tocdepth}{1}
\tableofcontents{}

\section{Introduction}

\subsection{Tautological rings of $\A_g$ and
\texorpdfstring{$\bar{\mathcal{A}}_g$}{}} 
Let $\mathcal{A}_g \subset \overline{\mathcal{A}}_g$ be a toroidal compactification of
the moduli space of principally polarized
abelian varieties.
The space $\A_g$ is a nonsingular Deligne-Mumford stack of dimension $\binom{g+1}{2}$, and the compactification $\overline{\mathcal{A}}_g$
is a reduced and irreducible (but possibly singular) proper Deligne-Mumford stack,
see \cite{FC}.
The Hodge bundle 
$$\mathbb{E} \rightarrow {\mathcal{A}}_g\, $$
is defined as the pullback to $\A_g$ via the zero section $s$ of the relative cotangent bundle
of the universal family $\mathcal{X}_g$ of abelian varieties,
$$p: \mathcal{X}_g \rightarrow \A_g\, , \ \ \ \ s: \A_g \rightarrow \mathcal{X}_g\, , \ \ \ \ \mathbb{E} \cong s^* \Omega_p\, .$$
There is a canonical extension of the Hodge bundle over $\overline{\mathcal{A}}_g$ by
\cite [Theorems V.2.3, VI.1.1, VI.4.2]{FC},
$$\mathbb{E} \rightarrow \overline{\mathcal{A}}_g\, .$$
By\cite{vdg, EV}, the Chern classes $\lambda_i$ of
the Hodge bundle
satisfy Mumford's relation{\footnote{Since $\overline{\mathcal{A}}_g$ is possibly singular, even as a stack, some care must be taken with the Chow theories. Here, 
$\CH^{\mathsf{op}}$ is the $\mathbb{Q}$-algebra of operational Chow classes. Usual Chow
cycle theory, indexed by codimension, is denoted by $\CH^*$.}}: 
\begin{equation}\label{mumford2}
 (1+\lambda_1+\lambda_2+\ldots+\lambda_g)(1-\lambda_1+\lambda_2-\ldots+(-1)^g\lambda_g)=1\, \in \,    \CH^{\mathsf{op}}(\overline{\A}_g)\, ,
\end{equation}
for all $g\geq 1$.
Van der Geer \cite{vdg} defined the tautological rings
$$\R^*({\mathcal{A}}_g)\subset \CH^*({\mathcal{A}}_g)\, , \ \ \ \
\R^*(\overline{\mathcal{A}}_g)\subset \CH^{\mathsf{op}}(\overline{\mathcal{A}}_g)$$ to be the $\qq$-subalgebras generated by the $\lambda$-classes.
Both tautological rings are calculated by a
fundamental result of \cite {vdg}. 

\begin{thm}[van der Geer]\label{vdgthm2}  The following
properties hold:
\begin{enumerate}
    \item[\textnormal{(i)}] The kernel
    of 
    the quotient
     $$\mathbb{Q}[\lambda_1,
\lambda_2, \lambda_3, \ldots,\lambda_g] \rightarrow
\mathsf{R}^*(\overline{\mathcal{A}}_g)\rightarrow 0$$
is generated as an ideal by
Mumford's  relation \eqref{mumford2}. 
   \item [\textnormal{(ii)}] $\R^*(\overline{\mathcal{A}}_g)$ is a Gorenstein local ring with socle in codimension $\binom{g+1}{2}$,
   $$\R^{\binom{g+1}{2}}(\overline{\mathcal{A}}_g) \cong 
   \qq\, .$$
   The class $\lambda_1\lambda_2 \lambda_3 \cdots \lambda_{g}$ is
   a generator of the socle.
   \item [\textnormal{(iii)}] $\R^*(\A_g)\cong \R^*(\overline{\A_g})/(\lambda_g)$ is a Gorenstein local ring with socle in codimension $\binom{g}{2}$,
   \[
   \R^{\binom{g}{2}}(\A_g)\cong \qq\,.
   \]
    The class $\lambda_1\lambda_2 \lambda_3 \cdots \lambda_{g-1}$ is a generator of the socle.

\end{enumerate}
\end{thm}

\subsection{Tautological projection for \texorpdfstring{$\overline{\A}_g$}{}}
The idea of tautological projection on $\overline{\A}_g$ (Definition \ref{compactprojection} below) appears in work of Faber \cite{Fa} and of Grushevsky and Hulek \cite {GH}.
Since 
$\overline{\mathcal{A}}_g$ is proper 
of dimension $\binom{g+1}{2}$, we obtain 
an
evaluation 
\begin{equation*} \epsilon^{\mathsf{cpt}}:\R^{\binom{g+1}{2}}(\overline{\mathcal{A}}_g) \to \mathbb Q\, , \ \quad \alpha \mapsto \int_{{\overline{\mathcal{A}}_g}} {\alpha} \, ,\end{equation*} and a pairing 
between classes 
on $\overline{\mathcal{A}}_g$, 
\begin{equation}\label{barpairing}
\langle\, ,\, \rangle^{\mathsf{cpt}}:\,
\mathsf{CH}^k(\overline{\mathcal{A}}_g)\times \mathsf{R}^{\binom{g+1}{2}-k}
(\overline{\mathcal{A}}_g)\,\to\,  \mathbb Q\, ,\ \quad \langle \gamma, \delta\rangle^{\mathsf{cpt}}= \int_{{\overline{\mathcal{A}}_g}} {\gamma}\cdot {\delta}\, .
\end{equation}
Here, $\mathsf{cpt}$ stands for compact.

By Theorem \ref{vdgthm2}(ii), the socle of $\R^*(\overline{\mathcal{A}}_g)$ is spanned by the class $\lambda_1\lambda_2\lambda_3\cdots \lambda_{g}$. Equivalently,
\begin{equation} \label{brrt}
\gamma_g=\epsilon^{\mathsf{cpt}}(\lambda_1\lambda_2\lambda_3 \cdots \lambda_g) \neq 0\,.
\end{equation}
The exact evaluation{\footnote{Here, $B_{2i}$ is the Bernoulli number.}},
\begin{equation}\label{exv}\gamma_g=\prod_{i=1}^{g} \frac{|B_{2i}|}{4i}\, ,\end{equation} 
is computed in \cite[page 9]{vdg}.
By the Gorenstein property of $\R^*(\overline{\mathcal{A}}_g)$, the pairing of tautological classes  
$$\R^k(\overline{\mathcal{A}}_g)\times \R^{\binom{g+1}{2}-k}(\overline{\mathcal{A}}_g)\, \to\,  \R^{\binom{g+1}{2}}(\overline{\mathcal{A}}_g) 
\, \cong\,  \mathbb Q$$
is non-degenerate (where the last isomorphism is via
 $\epsilon^{\mathsf{cpt}}$).

\begin{definition}\label{compactprojection} 
Let $\gamma \in \mathsf{CH}^*(\overline{\mathcal{A}}_g)$.
The {\em tautological projection} $\mathsf{taut}^{\mathsf{cpt}}(\gamma) \in \mathsf{R}^*(\overline{\mathcal{A}}_g)$ is the unique{\footnote{The existence and uniqueness of 
$\mathsf{taut}^{\mathsf{cpt}}(\gamma)$ follows from
the Gorenstein property of $\mathsf{R}^*(\overline{\mathcal{A}}_g)$ applied to the functional $\delta\mapsto \langle \gamma, \delta\rangle^{\mathsf{cpt}}$ on $\R^*(\overline{\A}_g)$.}}
tautological class which satisfies
$$\langle \mathsf{taut}^{\mathsf{cpt}}(\gamma),\delta\rangle^{\mathsf{cpt}} =
\langle \gamma,\delta \rangle^{\mathsf{cpt}} $$
for all classes $\delta\in \mathsf{R}^*(\overline{\mathcal{A}}_g)$.
\end{definition}

\noindent $\bullet$ If $\gamma \in \mathsf{R}^*(\overline{\mathcal{A}}_g)$, then
$\gamma = \mathsf{taut}^{\mathsf{cpt}}(\gamma)$, so we have a $\mathbb{Q}$-linear 
projection operator:
$$\mathsf{taut}^{\mathsf{cpt}}: \mathsf{CH}^*(\overline{\A}_g) \to \mathsf{R}^*(\overline{\A}_g)\, , \ \ \ 
\mathsf{taut}^{\mathsf{cpt}}\circ \mathsf{taut}^{\mathsf{cpt}}= 
\mathsf{taut}^{\mathsf{cpt}}\, .$$

\noindent $\bullet$ From the point of view of  $\A_g$,
a difficulty with the theory on $\overline{\A}_g$ is the dependence
upon compactification. Given a subvariety 
$$V \subset \A_g\, ,$$
we can define a projection
\begin{equation}
    \label{zz12}
\mathsf{taut}^{\mathsf{cpt}}([\overline{V}]) \in \mathsf{CH}^*(\overline{\A}_g)
\end{equation}
with respect to the Zariski closure $V\subset \overline{V}$ in a toroidal compactification $\A_g \subset \overline{\A}_g$,
but the projection \eqref{zz12}
will depend upon the choice of $\overline{\A}_g$.
In order to study cycles on the moduli of abelian varieties, we would 
like to construct a canonical projection operator depending just upon the interior $\A_g$.

\subsection{Top Chern class of the Hodge bundle}
To define a tautological projection operator
on $\A_g$, we will define a pairing similar to \eqref{barpairing}. 
The theory depends upon a new vanishing result for 
the top Chern class of Hodge bundle on $\overline{\A}_g$.

We recall the $\lambda_g$-pairing on tautological classes{\footnote{We refer
the reader to \cite{FP3,P} for a review of the theory of tautological
classes on the moduli spaces of curves. Unlike the case of $\A_g$,
the tautological ring $\R^*(\M_g^{\ct})$ is {\em not} a Gorenstein
local ring, see \cite{CLS, Pix}, and \cite{Pet} for the pointed case.}}on 
the moduli space of curves of compact type $\M_g^{\ct}\subset \Mb_g$. On the tautological ring of $\M_g^{\ct}$, the $\lambda_g$-pairing is given by 
\[
\R^k(\M_{g}^{\ct})\times \R^{2g-3-k}(\M_{g}^{\ct})\rightarrow \R^{2g-3}(\M_g^{\ct})\cong \qq\, , \ \quad (\alpha,\beta) \mapsto \int_{\Mb_g} \overline{\alpha}\cdot\overline{\beta}\cdot \lambda_g\, ,
\]
where $\overline{\alpha}$ and $\overline{\beta}$ are arbitrary lifts of $\alpha$ and $\beta$ to $\Mb_g$. The $\lambda_g$-pairing is well-defined, independent of the lifts, because $\lambda_g\in \R^g(\Mb_g)$ restricts trivially to the boundary $$\lambda_g |_{\Mb_g \smallsetminus \M_g^{\ct}} =0\, ,$$ 
see \cite {FP2}. 
A parallel boundary vanishing for $\lambda_g\in \R^*(\overline{\A}_g)$ is
our first result. 

\begin{thm}
\label{thm:main}
 The restriction of $\lambda_g$ to $\overline{\A}_g \smallsetminus \A_g$ 
vanishes for every toroidal compactification $\overline{\A}_g$.
\end{thm}

In characteristic $p$, the Theorem \ref{thm:main} can be proven{\footnote{We thank van der Geer for the characteristic $p$ argument. We do not know how to lift this argument to characteristic $0$.}}
by considering the $p$-rank zero locus in $\overline{\A}_g.$ The $p$-rank zero locus 
avoids the boundary and has class in $\R^*(\overline{\A}_g)$ equal
to a multiple of $\lambda_g$, \cite [Theorem 2.4]{vdg}. In all characteristics, the vanishing of $\lambda_g$ over the boundary of the partial compactification $\A_g^{\text{part}}$ of torus rank $1$ degenerations follows from the discussion of \cite [page 6]{vdg}. The statement for the entire boundary is new. 

Our proof of Theorem \ref{thm:main} is obtained as a consequence of the following statements about semistable degenerations of abelian varieties:
\begin{itemize}
    \item [(i)] The sheaf of relative log differentials has a trivial rank 1 quotient
    on the singularities of
    the fibers of the universal family. The trivial quotient statement is true for any semistable family, independently of abelian structure (also applying, for example,
    to  families of curves). 
    \item [(ii)] For abelian schemes, the sheaf of relative log differentials is isomorphic to the pullback of the Hodge bundle \cite {FC}.
    \end{itemize}
The full proof is presented in Section \ref{s2} after a review of log structures,
semistable degenerations, and residues. 

The proof of Theorem \ref{thm:main}
yields a stronger vanishing result: the top $m$
Chern classes of the Hodge bundle vanish on the locus 
$$\overline{\A}^{\geq m}_g \subset \overline{\A}_g$$
corresponding to semiabelian varieties of torus rank at least $m$, see Theorem 
\ref{thm: highervanishing}
of Section \ref{sss4}.

The vanishing of Theorem \ref{thm:main} of the top Chern class of the Hodge
bundle  for the moduli of abelian varieties implies the parallel vanishing 
for the moduli of curves, 
$$\lambda_g |_{\overline{\A}_g \smallsetminus \A_g} =0\ \ \ \implies \ \ \
\lambda_g |_{\Mb_g \smallsetminus \M_g^{\ct}} =0\, ,$$ 
via the Torelli map from 
$\Mb_g$ to a suitable{\footnote{ The second Voronoi
compactification can be taken here \cite{A1},\cite{N}.}} toroidal compactification $\overline{\A}_g$.
We can hope for an even deeper connection: a lifting of Pixton's formula
\cite{HMPPS, JPPZ}
for $\lambda_g$ on $\Mb_g$ to $\overline{\A}_g$. The natural
context for such a lifting should be the logarithmic Chow ring
of the moduli space of abelian varieties. A discussion of these
ideas is presented in Section \ref{ppp}.

\subsection{Tautological projection for \texorpdfstring{$\A_g$}{Ag}}
Given $\alpha\in \CH^*(\mathcal{A}_g)$, we define 
an
evaluation,
\begin{equation*} \epsilon^{\mathrm{ab}}:\CH^{\binom{g}{2}}(\mathcal{A}_g) \to \mathbb Q\, , \ \quad \alpha \mapsto \int_{\overline{\mathcal{A}}_g}  \overline{\alpha} \cdot \lambda_g\ = \ \deg (\lambda_g\cap \overline \alpha)\, ,\end{equation*} 
where $\overline \alpha$ is a lift to the toroidal compactification $\overline {\A}_g$. The answer is well-defined (independent of lift) by the
vanishing of Theorem \ref{thm:main}. The answer is also independent of the toroidal compactification (as can be seen by picking two arbitrary toroidal compactifications $\overline {\A}_g$ and $\overline{\A}'_g$ and passing to a third $\overline{\A}''_g$ which dominates both $\overline{\A}_g$ and $\overline{\A}'_g$).

We also have an induced pairing between classes 
on $\mathcal{A}_g$, \begin{equation}\label{par2}
\langle\, ,\, \rangle:\,
\mathsf{CH}^k(\mathcal{A}_g)\times \mathsf{R}^{\binom{g}{2}-k}(\mathcal{A}_g)\,\to\,  \mathbb Q\, ,\ \quad \langle \gamma, \delta\rangle=\int_{{\overline{\mathcal{A}}_g}}\overline{\gamma}\cdot \overline{\delta}\cdot \lambda_g=\deg (\overline \delta \cdot \lambda_g\cap \overline \gamma)\, ,\end{equation}
which we call the {\em $\lambda_g$-pairing}
for the moduli of abelian varieties. Here, $\overline \delta$ is a lift of $$\delta\in \R^*(\A_g)=\R^*(\overline{\mathcal{A}}_g)/\langle \lambda_g\rangle$$ to a tautological class on $\overline{\mathcal{A}}_g,$ while $\overline \gamma$ is an arbitrary lift of $\gamma$. The lift of $\delta$ is well-defined up to the class $\lambda_g$, while the lift of $\gamma$ is well-defined up to cycles supported on the boundary. 
The vanishing of $\lambda_g^2$ on 
$\overline{\A}_g$ and 
Theorem \ref{thm:main}  ensure that the
$\lambda_g$-pairing is well-defined. As in the previous paragraph, the $\lambda_g$-pairing is independent of the choice of compactification $\A_g\subset \overline{\A}_g.$

By Theorem \ref{vdgthm2}(iii), the socle of $\R^*(\A_g)$ is spanned by the class $\lambda_1\lambda_2\lambda_3\cdots \lambda_{g-1}$.  
The Gorenstein property of $\R^*(\A_g)$ together with the
non-vanishing \eqref{brrt} implies
that the restriction of the $\lambda_g$-pairing to tautological classes  
$$\R^k(\A_g)\times \R^{\binom{g}{2}-k}(\A_g)
\to \R^{\binom{g}{2}}(\A_g) 
\cong \mathbb Q$$
is non-degenerate (where the last isomorphism is via
 $\epsilon^{\mathsf{ab}}$).

\begin{definition} \label{cvvt}
Let $\gamma \in \mathsf{CH}^*(\mathcal{A}_g)$.
The {\em tautological projection} $\mathsf{taut}(\gamma) \in \mathsf{R}^*(\mathcal{A}_g)$ is the unique{\footnote{The existence and uniqueness of $\mathsf{taut}(\gamma)$ follows from
the Gorenstein property of $\mathsf{R}^*(\mathcal{A}_g)$.}}
tautological class which satisfies
$$\langle \mathsf{taut}(\gamma),\delta\rangle =
\langle \gamma,\delta \rangle $$
for all classes $\delta\in \mathsf{R}^*(\mathcal{A}_g)$.
\end{definition}

\noindent $\bullet$ If $\gamma \in \mathsf{R}^*(\mathcal{A}_g)$, then
$\gamma = \mathsf{taut}(\gamma)$, so we have a $\mathbb{Q}$-linear 
projection operator:
$$\mathsf{taut}: \mathsf{CH}^*(\A_g) \to \mathsf{R}^*(\A_g)\, , \ \ \ 
\mathsf{taut}\circ \mathsf{taut}= \mathsf{taut}\, .$$

\noindent $\bullet$ For $\gamma \in \mathsf{CH}^*({\mathcal{A}}_g)$, 
tautological projection provides a canonical decomposition 
$$\gamma = \mathsf{taut}(\gamma)  + (\gamma -
\mathsf{taut}(\gamma))
$$
into purely tautological and purely non-tautological parts. 

\noindent $\bullet$ Tautological projection commutes with restriction: 
for every toroidal compactification $\A_g \subset \overline{\A}_g$ and 
every class  $\gamma\in \mathsf{CH}^*(\overline{\A}_g),$
$$\mathsf{taut}^{\mathsf{cpt}}(\gamma)\big|_{\A_g}=\mathsf{taut}\left(\gamma\big|_{\A_g}\right)\, .$$ 
To prove the restriction property, consider classes $$\gamma \in \CH^*(\overline{\A}_g)\ \ \text{and}\ \ \delta \in \R^*({\A}_g)\, .$$
Equations \eqref{barpairing} and \eqref{par2} imply the compatibility between pairings \begin{equation}\label{compp}\langle\gamma\big|_{\A_g}, \delta\rangle=\langle \gamma, \bar \delta \,\lambda_g\rangle^{\textsf{cpt}},\end{equation} where $\bar \delta$ is any lift of $\delta$ to the compactification $\overline{\A}_g.$ Then, 
\begin{eqnarray*}
\langle \mathsf{taut}^{\mathsf{cpt}}(\gamma), \bar \delta \,\lambda_g\rangle^{\mathsf{cpt}}=\langle \gamma, \bar \delta \lambda_g\rangle^{\mathsf{cpt}}\ &\implies& \ \langle \mathsf{taut}^{\mathsf{cpt}}(\gamma)\big|_{\A_g}, \delta\rangle=\langle \gamma\big|_{\A_g},\delta\rangle \end{eqnarray*}
which implies $\mathsf{taut}^{\mathsf{cpt}}(\gamma)\big|_{\A_g}=\mathsf{taut} \left(\gamma\big|_{\A_g}\right).$
Here, we have used Definition \ref{compactprojection}, equation \eqref{compp}, and Definition \ref{cvvt} (and the argument is not possible
without Theorem \ref{thm:main}).

\subsection{Tautological projection of product classes}
As an application of the theory, we consider the tautological projections of product loci. For $g=g_1+g_2$ with $g_i\geq 1$,
the product map $$\A_{g_1}\times \A_{g_2}\rightarrow \A_{g}$$
defines a class $[\A_{g_1}\times \A_{g_2}]\in \mathsf{CH}^*(\A_g)$ by pushforward of the fundamental cycle. 
More generally,
for every partition $g=\sum_{i=1}^{\ell} g_i$ in positive parts,
we have a product map and an associated class:
\begin{equation} \label{pprr}
\prod_{1=1}^\ell \A_{g_i} \rightarrow \A_g\, , \ \ \ \ \ \
\left[\prod_{1=1}^\ell \A_{g_i}\right] \in \mathsf{CH}^*(\A_g)\, .
\end{equation}

Whether these product maps and classes \eqref{pprr}
naturally extend to a compactification 
$\overline{\A}_g$
depends upon the choice of toroidal compactification. 
Toroidal compactifications of $\A_g$ correspond to choices of admissible fans $\Sigma_g$ in $\Sym^2_{\mathrm{rc}}(\mathbb{R}^g)$, the rational closure of the positive-definite symmetric forms on $\mathbb R^g$. Following \cite{GHT}, a collection of such fans $\{\Sigma_g\}_{g\in \mathbb{N}}$ is {\em additive} if the sum $\sigma_1\oplus \sigma_2$ of any two cones $\sigma_1\in \Sigma_{g_1}$ and $\sigma_2\in \Sigma_{g_2}$ is a cone in $\Sigma_{g_1+g_2}$. Let $\overline{\A}^{\Sigma_g}_g$ be a toroidal compactification corresponding to an additive collection of fans $\{\Sigma_g\}$. The {\em perfect cone} compactification  satisfies these properties, see \cite{SB}. In
the additive case, the product maps extend,
$$\prod_{i=1}^\ell \overline{\A}^{\Sigma_{g_i}}_{g_i} \rightarrow \overline{\A}^{\Sigma_g}_g\, ,$$
and we can therefore define cycles 
$$\left[\prod_{i=1}^\ell \overline{\A}^{\Sigma_{g_i}}_{g_i}\right] \in \mathsf{CH}^*(\overline{\A}_g^{\Sigma_g})\, .$$

While the definition of tautological projection is
independent of toroidal compactification, natural compactifications
can be used for the calculation. We prove a closed formula for the tautological projection of the product cycles. The result extends calculations of \cite{GH} for $g\leq 5$. 

\begin{thm}\label{t25} For $g_1+\ldots+g_\ell=g$, the tautological projection of the product locus $\overline{\A}_{g_1}^{\Sigma_{g_1}}\times\cdots \times\overline{\A}_{g_\ell}^{\Sigma_{g_\ell}}$ in $\overline{\A}_g^{\Sigma_g}$ is given by a $g\times g$ determinant, $$\mathsf{taut}^{\mathsf{cpt}} ([\overline{\A}_{g_1}^{\Sigma_{g_1}}\times\cdots \times\overline{\A}_{g_\ell}^{\Sigma_{g_\ell}}])= \frac{\gamma_{g_1}\cdots\gamma_{g_\ell}}{\gamma_g} \begin{vmatrix}\lambda_{\alpha_1} & \lambda_{\alpha_1+1} & \ldots & \lambda_{\alpha_1+g-1} \\ \lambda_{\alpha_2-1} & \lambda_{\alpha_2} & \ldots & \lambda_{\alpha_2+g-2} \\ \ldots & \ldots & \ldots & \ldots \\ \lambda_{\alpha_g-g+1} & \lambda_{\alpha_g-g+2} & \ldots & \lambda_{\alpha_g}\end{vmatrix}\, , $$ for the vector $$\alpha = (\underbrace{g-g_1, \ldots, g-g_1}_{g_1}, \underbrace{g-g_1-g_2, \ldots, g-g_1-g_2}_{g_2}, \ldots, \underbrace {g-g_1-\ldots-g_{\ell}, \ldots, g-g_1-\ldots-g_{\ell}}_{g_{\ell}}).$$ We set $\lambda_k=0$ for $k<0$ or $k>g$ and $\lambda_0=1.$

\end{thm}

\noindent In the above determinant, $\alpha_i$ denotes the $i^{\text{th}}$ component of the vector $\alpha$. The last $g_\ell$ entries of $\alpha$ are $0$, and contribute rows with $1$'s on the main diagonal and $0$'s below the main diagonal. These last entries do not change the determinant, but are included for a more symmetric formulation. The constants $\gamma_g$ are defined in \eqref{exv}.

The proof of Theorem \ref{t25} in Section \ref{a1} relies on the connections between the tautological ring of $\overline {\A}_g^{\Sigma_g}$ and the Chow ring of the Lagrangian Grassmannian $\LG_g$ of $\mathbb C^{2g}$ as explained in  \cite{vdg2}. The argument combines properties of tautological projection, the Hirzebruch-Mumford proportionality principle, and the geometry of $\LG_g$.

Using the restriction property of
tautological projection  and the relations in Theorem \ref{vdgthm2} (iii),
we prove the following result in Section \ref{agres}.

\begin{thm} \label{aagg} For $g_1+\ldots+g_\ell=g$, the tautological projection of the product locus ${\A}_{g_1}\times\cdots \times {\A}_{g_\ell}$ in ${\A}_g$ is given by a $(g-\ell) \times (g-\ell)$ determinant, 
$$\mathsf{taut}{\left([\A_{g_1}\times \cdots \times \A_{g_\ell}]\right)}=\frac{\gamma_{g_1}\cdots\gamma_{g_\ell}}{\gamma_g}\cdot \lambda_{g-1}\cdots \lambda_{g-\ell+1}\cdot \begin{vmatrix} \lambda_{\beta_1} & \lambda_{\beta_1+1} & \ldots & \lambda_{\beta_1+g^*-1} \\ \lambda_{\beta_2-1} & \lambda_{\beta_2} & \ldots & \lambda_{\beta_2+g^*-2} \\ \ldots & \ldots & \ldots & \ldots \\ \lambda_{\beta_{g^*}-g^*+1} &\lambda_{\beta_{g^*}-g^*+2} & \ldots & \lambda_{\beta_{g^*}}\end{vmatrix}\, ,$$ 
for the vector $\beta$ given by 
$$(\underbrace{g^*-g^*_1, \ldots, g^*-g^*_1}_{g^*_1}, \underbrace {g^*-g^*_1-g^*_2, \ldots, g^*-g^*_1-g^*_2}_{g_2^*}, \ldots, \underbrace{g^*-g_1^*-\ldots-g_{\ell}^*,\ldots,
g^*-g_1^*-\ldots-g_{\ell}^*}_{g^*_\ell})\,,$$
where $g^*=g-\ell$ and $g_i^*=g_i-1$. 
\end{thm}

 The tautological projections of the product loci 
in $\A_g$ from Theorem \ref{aagg}
 in the simplest cases are:
        \begin{equation}\label{a1smooth}
    \mathsf{taut}\left([{\A}_1\times \A_{g-1}]\right)=\frac{g}{6|B_{2g}|} \lambda_{g-1}\, ,
\end{equation}
\begin{equation}\label{a2smooth}\mathsf{taut}\left(\left[\A_{2}\times\A_{g-2}\right]\right)=\frac{1}{360}\cdot \frac{g(g-1)}{|B_{2g}| |B_{2g-2}|}\cdot \lambda_{g-1}\lambda_{g-3}\,,\end{equation}

\begin{equation}\label{a3smooth}\mathsf{taut}\left(\left[\A_{3}\times\A_{g-3}\right]\right)=\frac{1}{45360}\cdot \frac{g(g-1)(g-2)}{|B_{2g}||B_{2g-2}||B_{2g-4}|}\cdot \lambda_{g-1}(\lambda_{g-4}^2-\lambda_{g-3}\lambda_{g-5})\, ,\end{equation}

\begin{equation}\label{a1a2smooth}\mathsf{taut}\left(\left[\A_{1} \times
\A_2
\times\A_{g-3}\right]\right)=\frac{1}{2160}\cdot \frac{g(g-1)(g-2)}{|B_{2g}||B_{2g-2}||B_{2g-4}|}\cdot \lambda_{g-1}\lambda_{g-2}\lambda_{g-4}\, ,\end{equation}

\begin{equation}\label{tp}\mathsf{taut}\left(\left[\underbrace{\A_1\times\ldots \times \A_1}_{k}\times \A_{g-k}\right]\right)=\left(\prod_{i=g-k+1}^{g} \frac{i}{6|B_{2i}|}\right)\lambda_{g-1}\cdots \lambda_{g-k}\, .\end{equation} 
In genus $g=4$, formula \eqref{tp} yields 
$$\mathsf{taut}\left(\left[\A_1\times \A_1 \times \A_{2}\right]\right)=420\lambda_3\lambda_2\, , \quad \mathsf{taut}\left(\left[\A_1\times \A_1 \times \A_1\times \A_1 \right]\right)=4200\lambda_3\lambda_2\lambda_1\, .$$ In fact, $\left[\A_1\times \A_1 \times \A_{2}\right]$ and $\left[\A_1\times \A_1 \times \A_1\times \A_1 \right]$ are tautological \cite{COP}. 

An interesting case of \eqref{tp} occurs when $k=g-1$ since the class $\lambda_{g-1}\cdots \lambda_{1}$ 
generates the socle of the tautological ring $\R^{\binom{g}{2}}(\A_g)$. 
A speculation of \cite {COP} is that the $g$-fold product 
$$[\A_1\times\cdots\times\A_1]\in \mathsf{CH}^{\binom{g}{2}}(\A_g)$$ 
also lies in the socle of the tautological ring. If the speculation is correct, then we obtain an exact evaluation \begin{equation}\label{tp1}[\A_1\times\dots\times\A_1]=\left(\prod_{i=1}^{g} \frac{i}{6|B_{2i}|}\right) \lambda_{g-1}\cdots \lambda_{1}\, .\end{equation}

\vspace{4pt}

\noindent {\bf Question A.} {\em When is the non-tautological part
of the product locus nonzero: 
$$
\left[\prod_{1=1}^\ell \A_{g_i}\right] -
\mathsf{taut}\left(\left[\prod_{1=1}^\ell \A_{g_i}\right] \right) \neq 0 \, ?
$$}
\vspace{6pt}

The cycles $[\A_1\times \A_{g-1}]\in \CH^*(\A_g)$ are studied in \cite {COP} via the Torelli map $${\Tor}: \mathcal M_g^{\ct}\to \A_g\, .$$ A central
result of \cite{COP} is that
$[\A_1\times \A_5]$ is not tautological on $\A_6$, so
the purely non-tautological part is nonzero,
$$
\left[\A_1 \times \A_5\right] -
\mathsf{taut}\left(\left[\A_1 \times \A_5\right] \right) \neq 0 \, .
$$
Detection of the non-vanishing of the non-tautological part is
subtle since the class 
$$\Delta_g={\Tor}^*\left([\A_1\times \A_{g-1}]-\frac{g}{6|B_{2g}|} \lambda_{g-1}\right)$$ lies in the kernel of the $\lambda_g$-pairing $$\R^{g-1}(\mathcal M_g^{\ct})\times \R^{g-2}(\mathcal M_g^{\ct})\to \R^{2g-3}(\mathcal M_g^{\ct})\cong \mathbb Q$$
for all genera $g$ by an argument of Pixton, see \cite{COP}.

\vskip.1in

The Noether-Lefschetz loci in $\A_g$ parametrize abelian varieties whose Picard rank jumps. The Noether-Lefschetz loci have been classified in \cite {DebLas} (the products $\A_{g_1}\times \A_{g_2}$ for $g_1+g_2=g$ arise in the classification, but
there are other loci as well). 
The Appendix contains a conjecture
about the projection of the
Noether-Lefschetz locus of abelian varieties with real multiplication.
We can hope
for a more general result. \vskip.1in

\noindent {\bf Question B.} {\em Calculate the tautological projections of all Noether-Lefschetz loci in $\A_g$.}
\vspace{8pt}

Beyond product and Noether-Lefschetz cycles, we can consider the tautological projection of the locus of Jacobians 
of genus $g$ curves of compact type,
$$[\mathcal{J}_g]= 
\Tor_*[\mathcal{M}_g^{\ct}]
\in \CH^*(\A_g)\, .$$ 
Faber \cite{Fa} determined these explicitly for $g\leq 7$. For all genera, in the basis of monic square-free monomials in the $\lambda's$, the leading term is given by $$\mathsf{taut}\left(\left[\mathcal {J}_g\right]\right)=\left(\frac{1}{g-1}\prod_{i=1}^{g-2}\frac{2}{(2i+1)|B_{2i}|}\right)\lambda_1\cdots \lambda_{g-3}+\ldots\, ,$$ 
as proposed in 
\cite [Conjecture 1]{Fa} and  confirmed via \cite [Theorem 4]{FP1}. 
A more complicated formula for the coefficient of the term $\lambda_2\ldots \lambda_{g-4} \lambda_{g-2}$ was predicted by \cite [Conjecture 2]{Fa} and subsequently proven in \cite [Section 5.2]{FP2}.

 For each genus $g$, the class
$\mathsf{taut}\left(\left[\mathcal {J}_g\right]\right)
 \in \mathsf{R}^*(\A_g)
$
can be computed algorithmically by a finite number of Hodge integral
evaluations \cite{Fa}.
Finding expressions for the coefficients of the remaining terms of $\mathsf{taut}\left(\left[\mathcal {J}_g\right]\right)$
is an open question, but we could hope for more structure.

\vspace{8pt}
\noindent {\bf Question C.} {\em Is there a simpler
way to understand the tautological
projection 
$$ \mathsf{taut}\left(\left[\mathcal {J}_g\right]\right) \in \mathsf{R}^*(\A_g)\,? $$
When is the non-tautological part nonzero?}

\subsection*{Acknowledgments}
We thank Valery Alexeev, 
Gerard van der Geer, Fran\c{c}ois Greer, Sam Grushevsky, Klaus Hulek,  Carl Lian,
Aitor Iribar L\'opez, Andrew Kresch, and
Aaron Pixton for several related conversations. We thank the referee for a careful reading.

Discussions with Carel Faber have played a crucial role in our work.
His early calculations of tautological projections on  toroidal
compactifications include an independent (unpublished) proof of
formula \eqref{a1smooth}. After a lecture by R.P. at the {\em Belgian-Dutch 
    Algebraic Geometry Seminar} in Leiden in November 2023, Ben Moonen suggested
    another path to the $\lambda_g$ vanishing of Theorem \ref{thm:main}
    by applying rigidity results on the boundary of $\overline{\A}_g$.
    We have chosen to use an approach 
    using semistable reduction and
    residue maps. We thank Dan Abramovich, Younghan Bae, and Dori Bejleri for discussions about residues and for comments on earlier versions of the construction. 

    S.C. was supported by a Hermann-Weyl-Instructorship from the Forschungsinstitut f\"ur Mathematik at ETH Z\"urich.
S.M. was supported by 
SNF-200020-182181 and ERC-2017-AdG-786580-MACI. R.P. was supported by
SNF-200020-182181, SNF-200020-219369, ERC-2017-AdG-786580-MACI, and SwissMAP. 

This project has
received funding from the European Research Council (ERC) under the Euro-
pean Union Horizon 2020 research and innovation program (grant agreement No. 786580).

\section{The top Chern class of the Hodge bundle}\label{s2}

\subsection{Overview}
Logarithmic geometry provides a convenient tool for considering all toroidal compactifications of $\A_g$ simultaneously and plays an important
role in the proof of Theorem \ref{thm:main}. A quick review of the basic language of log geometry is given in Section \ref{sss1}.
The proof of Theorem \ref{thm:main} relies on the residue map constructed in Section \ref{sss2}. An analogous residue map was constructed by Esnault-Viewheg \cite{EV}. We give a different perspective on the construction.
After a discussion of the Hodge bundle on toroidal compactifications of $\A_g$ in Section \ref{sss3},
the proof of Theorem \ref{thm:main} is presented in Section \ref{sss4}. Conjectures
and future directions related to tautological classes on toroidal compactifications of $\A_g$
are discussed in Section \ref{ppp}.

\subsection{The logarithmic Chow ring for toroidal embeddings} 
\label{sss1}

We will use the  language of log geometry and assume some rudimentary familiarity with the central definitions as given in \cite{Kato}. A summary of the relevant background information can be found in \cite{MPS}.

For a log scheme $(S,M_S)$, we write 
\[
\epsilon: M_S \to \mathcal{O}_S
\]
for the structure morphism
from the monoid $M_S$.
Let $M_S^\gp$ be the associated 
group, and let $\overline{M}_S$ be the characteristic monoid 
\[
\overline{M}_S = M_S/\mathcal{O}_S^*\, .
\]
The sheaf $\overline{M}_S$ is constructible, and thus stratifies $S$. 

For a toroidal embedding $(S,D)$, the log structure is given by the \'etale sheaf of monoids
\[
M_S = \{f \in \mathcal{O}_S: f \textup{ is a unit on }S\smallsetminus D\}\, .
\]
For toroidal embeddings, we will denote the
log structure by either
$(S,D)$
or $(S,M_S)$ depending upon context.

An important special case is that of a normal crossings pair $(S,D)$: a smooth Deligne--Mumford stack $S$ with a normal crossings divisor $D$. These are precisely the log smooth log Deligne--Mumford stacks with smooth underlying stack. The normal crossings condition is equivalent to 
\[
\overline{M}_{S,s} = \NN^k
\]
for each $s \in S$ and for some $k$ depending on $s$. The integer $k$ is the number of local branches of $D$ passing though $s$. 

{For a morphism of log structures $f:X \to S$,  let
\[
\overline{M}_{X/S} = \overline{M}_X/\overline{M}_S
\]
be the relative characteristic monoid. The morphism $f$ is \emph{strict} if $\overline{M}_{X/S}=0$. For a log scheme $(X,M_X)$ defined over a field $K$, a (global) \emph{chart} for the log structure of $X$ is a finitely generated, saturated monoid $P$ and a strict map $$X \to \textup{Spec}(K[P])\, ,$$
where the target carries the canonical log structure (coming from the torus invariant divisor). We require all log schemes to admit charts \'etale locally. 

A morphism $f:X \to S$ is \emph{log smooth} if, \'etale locally on $X$, we can find a map of monoids $Q \to P$ such that $\text{Ker} (Q^{\gp}\to P^{\gp})$ and the torsion part of $\text{Coker} (Q^{\gp}\to P^{\gp})$ are finite groups of order invertible in $K$, and a diagram}
\[
\begin{tikzcd}
    X \ar[r,"\alpha_X"] & S \times_{\Spec K[Q]} \Spec K[P] \ar[r] \ar[d] & \Spec K[P] \ar[d] \\ 
    & S \ar[r,"\alpha_S"] & \Spec K[Q]
\end{tikzcd}
\]
{with $\alpha_X,\alpha_S$ strict and  $\alpha_X$ smooth. If we can find such a diagram with $Q \to P$ of finite index and $\alpha_X$ \'etale, the morphism is {\em log \'etale}. In particular, toroidal embeddings $(S,D)$ are exactly the log smooth log schemes over $(\Spec K,K^*)$.}

To a toroidal embedding $(S,D)$, we can associate a cone complex $\Sigma_{(S,D)}$. We refer the reader to \cite[Section 4.3]{MPS} for an outline of the construction and further references. Each cone has an integral structure, and the cone complex is built by gluing the cones together with their integral structure. A {\em strata blowup} is a blowup of $(S,D)$ along a smooth closed stratum. The result is a new toroidal embedding $(S',D')$ with $D'$ the total transform of $D$, so the procedure can be iterated indefinitely. A \emph{log modification} of $(S,D)$ is a proper birational map $S' \to S$ that can be dominated by an iterated strata blowup. More intrinsically, the log modifications of $S$ are precisely the proper, representable, log \'etale monomorphisms $S' \to S$. Combinatorially, log modifications of $S$ correspond exactly to subdivisions of the cone complex $\Sigma_{(S,D)}$.

Log modifications form a filtered system. Indeed, two log modifications $$S' \to S\ \ \ \text{and} \ \ \ S'' \to S$$ can always be dominated by a third: the log modification corresponding to the common refinement of the subdivisions corresponding to $S'$ and $S''$ together with the intersection of the integral structures. To a toroidal embedding $(S,D)$, we can thus associate refined operational Chow groups
\[
\mathsf{log CH}^*(S,D) = \varinjlim \mathsf{CH}^{\textup{op}}(S')\, ,
\]
where $S'$ ranges over log modifications of $S$. 

Toroidal compactifications of $\A_g$ correspond to admissible decompositions of the rational closure $\Sym^2_{\mathrm{rc}}(\mathbb{R}^g)$ of the cone of positive-definite symmetric quadratic forms on $\mathbb{R}^g$. Any two admissible decompositions can be refined by a third. Hence, $\mathsf{log CH}^*(\overline{\A}_g,\partial \overline{\A}_g)$ is independent of the choice of compactification. We define 
\[
\mathsf{log CH}^*(\A_g)=\mathsf{log CH}^*(\overline{\A}_g,\partial \overline{\A}_g)
\]
for any toroidal compactification $\overline{\A}_g$.

\subsection{Semistable families and residues} \label{sss2}
For suitable families of log schemes,
we prove the  existence of a residue map in 
Theorem \ref{thm: residue} below.
The residue map will be applied in Section \ref{sss3}
to the universal family over the
moduli space of abelian varieties
in order to prove Theorem \ref{thm:main}.

The sheaf of relative \emph{logarithmic differentials}  $\Omega_{X/S}^{\log}$ is defined as the quotient of 
\[
\Omega_{X/S} \oplus (\mathcal{O}_X \otimes_{\mathbb{Z}}  M_X^\gp)
\]
by the subsheaf locally generated by sections of the form 
\begin{align*}
(d\epsilon(m),0)-(0, \epsilon(m)\otimes m)\,  \ \ \ \text{and} \ \ \  
(0,1 \otimes n) \, ,
\end{align*}
where $m\in M_X$ and 
$n\in \text{Im}(M_S^\gp)\subset M^\gp_X$, see \cite{Kato}. As usual, we write 
\[
d\log m = (0, 1 \otimes m)\,\]
which we view as
$d\epsilon(m)/\epsilon(m)$. 

For a strict map $f:X \to S$, we have $\Omega_{X/S}^{\log}= \Omega_{X/S}$, and for a log \'etale map $f:X \to S$, we have $\Omega_{X/S}^{\log}=0$.

\begin{definition}
\label{def:residues}
The \emph{sheaf of residues}  is defined to be the quotient 
    \[
    \mathcal{R}=\Omega_{X/S}^{\log}/\Omega_{X/S}\, .
    \]
    \end{definition}

\begin{definition}{(\hspace{-1pt}\protect{\cite[Definition 2.1.2]{Mss}})}
\label{def:logfamily}
    A \emph{logarithmic family} $X \to S$ is a log smooth, surjective, integral and saturated map of log schemes. 
\end{definition}

\noindent Families of stable curves and families of toroidal compactifications of semi-abelian schemes are all examples of log families. The condition that $f$ is integral and saturated -- called weak semistability in \cite{Mss} -- is a technical condition that, for log smooth $f$, implies that $f$ is flat with reduced fibers \cite [Lemma 3.1.2]{Mss}, \cite [Theorem II.4.2]{Tsuji}.{\footnote{When $S$ is also log smooth, the integrality condition is equivalent to $f$ being flat, and the conditions that $f$ is integral and saturated together are equivalent to $f$ being flat with reduced fibers.}} For a thorough discussion, see \cite[Part III, Section 2.5]{Ogus}. Being integral and saturated is local on $X$ and can be understood in terms of the cone complexes $\Sigma_X$ and $\Sigma_S$.
Integrality combined with saturatedness says, locally on $X$, that the associated map $\Sigma_X \to \Sigma_S$ maps cones of $\Sigma_X$ surjectively onto cones of $\Sigma_S$ and that the integral structure of a cone $\sigma$ surjects onto the integral structure of its image cone.  

Given a log scheme $(S,M_S)$ and a finite index extension of sheaves $M_S \to M_S'$, there is a universal log DM stack $(S',M_{S}')$ with a log map to $(S,M_S)$ whose map on log structures is given by the extension $M_S \to M_S'$. The stack $S'$ is called the root of $S$ along $M_S \to M_S'$. The simplest instance of this operation is taking a root along a boundary stratum of a normal crossings pair $(S,D)$. We call a composition of logarithmic modifications and roots a \emph{logarithmic alteration}. Log alterations of toroidal embeddings are isomorphisms on $S\smallsetminus D$, but are not necessarily representable. Logarithmic alterations are furthermore log \'etale. See
\cite{MWsheaf} for a lengthier discussion.

\begin{rem}\label{r8}
    Because we work with $\qq$-coefficients, pullbacks via  root maps induce isomorphisms on Chow groups. Therefore, for a toroidal embedding $(S,D)$, the logarthmic Chow groups can be equivalently defined as
    \[
    \mathsf{log CH}^*(S,D)=\varinjlim \mathsf{CH}^{\textup{op}}(S')\, , 
    \]
    where $S'$ ranges over logarithmic alterations of $S$.
\end{rem}

\begin{definition}
Let $f:X \to S$ be a log map. A \emph{logarithmic alteration of $f$} is a log map $f':X' \to S'$ and a commutative diagram
    \[
    \begin{tikzcd}
X' \arrow[d, "f'"] \arrow[r] & X \arrow[d, "f"] \\
S' \arrow[r]                 & S               
\end{tikzcd}
    \]
    such that $S' \to S$ and $X' \to X$ are logarithmic alterations. 
\end{definition}

\begin{thm}[Semistable Reduction Theorem, \cite{ALT,AK,Mss}]\label{thm:semistablereduction}
Let $f: X \to S$ be a dominant log smooth morphism of logarithmic schemes. Then there is a log alteration $f':X' \to S'$ of $f$ which is a log family. Furthermore, if $S$ is log smooth, then one can take $X'$ and $S'$ smooth.   
    \end{thm}

\begin{definition}
A pair $(X,D)$ is called \emph{simplest normal crossings} if $D\subset X$ is normal crossings in the Zariski topology and each intersection of components of $D$ is connected.
\end{definition}

\begin{rem}
    \label{rem:monodromy} In more geometric terms, a toroidal embedding $(X,D)$ has simplest normal crossings if the following conditions are all satisfied:
    \begin{enumerate} 
    \item [(i)] $(X,D)$ is a normal crossings pair (in particular, $X$ is nonsingular),
    \item [(ii)] the components of $D$ have no self-intersection,
    \item [(iii)] intersections of components of $D$ are connected. 
    \end{enumerate}
    We note that properties (ii) and (iii) are stable under logarithmic alterations. On the other hand, a logarithmic alteration of a normal crossings pair $(X,D)$ can be singular, so property (i) is not stable. Any log alteration $(X',D') \to (X,D)$ of a toroidal embedding that satisfies properties (i), (ii), and (iii) will also satisfy properties (i), (ii), and (iii) as long as $X'$ is nonsingular.
\end{rem}

    \begin{cor}
    \label{cor:simplestncreduction}
        Let $f: (X,D_X) \to (S,D_S)$ be a log family with $(S,D_S)$ toroidal. Then there is a log alteration $f': (X',D_{X'}) \to (S',D_{S'})$ which is a log family where both $S'$ and $X'$ have simplest normal crossings.  
    \end{cor}
\begin{proof}
    We can make $S$ and $X$ simplest normal crossings by suitable log modifications 
    $$S_1 \to S\,, \ \ \  X_1 \to X \times_S S_1\, .$$ For example, we can take the log modifications corresponding to double barycentric subdivisions (see the discussion in \cite[Section 5.6]{MPS}). Since logarithmic alterations are log \'etale and surjective, the morphism $X_1 \to S_1$ remains log smooth and surjective. Moreover, as noted in Remark \ref{rem:monodromy}, any further log alteration of $X_1$ or $S_1$ which is smooth will have simplest normal crossings. Therefore, we may apply Theorem \ref{thm:semistablereduction} to $X_1 \to S_1$ to get the desired $X' \to S'$.  
\end{proof}

Let $(X,M_X)$ be a log scheme. An \emph{orientation} on $M_X$ is an ordering of the irreducible elements\footnote{An element of a monoid is irreducible if it cannot be written as a sum of two non-zero elements.} of $\overline{M}_X(U)$ for all $U \subset X$ compatible with the restriction maps. We say $M_X$ is \emph{orientable} if it admits an orientation.

\begin{lem}
\label{lem:orient}
    Suppose $(X,D)$ is a pair with simplest normal crossings. Then the log structure of $(X,D)$ is orientable.  
\end{lem}

\begin{proof}
    We choose an ordering of the components $D_i$ of the divisor $D$. Since $\overline{M}_X$ is a locally constant constructible sheaf, giving an ordering of $\overline{M}_X$ is equivalent to giving an ordering of its stalks compatible with specializations of points. Every stratum of $(X,D)$ is the intersection 
    \[
    D_{i_1} \cap \cdots \cap D_{i_k}
    \]
    where the components appear in the ordering we have chosen. For each $x$ in the stratum, we have a canonical isomorphism 
    \[
    \overline{M}_{X,x} = \bigoplus \NN m_{i_k}\, .
    \]
    Here, $m_{i_k}$ is the image in $\overline{M}_{X,x}$ of any element $\widetilde{m}_{i_k} \in M_{X,x}$ that maps to a local defining equation for $D_{i_k}$.  We order the sections as $m_{i_1} < m_{i_2} \cdots < m_{i_k}$. This ordering is clearly compatible with specializations of points $x \rightsquigarrow y$, as it is induced from the global ordering of the divisors $D_i$.
\end{proof}

\begin{thm}
\label{thm: residue}
    Let $f:X \to S$ be a log family with $X$ and $S$ simplest normal crossings pairs. Every choice of orientation on $\overline{M}_X$ yields a map 
    \begin{equation}\label{resmap}
    \mathcal{R} \to \oplus_{H} \mathcal{O}_H
    \end{equation}
    where $H$ ranges over the irreducible components of the the locus where $\textup{rank } \overline{M}_{X/S} \ge 1$. Furthermore, the projection 
    \[
    \mathcal{R} \to \mathcal{O}_H
    \]
    of \eqref{resmap} to each summand $\mathcal{O}_H$ is surjective. 
\end{thm}
\begin{proof}
    
    Choose an orientation of $\overline{M}_X$ as in Lemma \ref{lem:orient}. We will construct the map \eqref{resmap} locally on $X$, and then we will 
    prove the gluing compatibility required for the  global 
    definition.

    Let $x$ be a point of $X$. Choose an ordered basis 
    $$m_1,\dots,m_n\in \overline{M}_{X,x}$$ 
    of $\overline{M}_X$ at $x$. 
    Let $\widetilde{m}_i$ be arbitrary lifts in $M_X$, and write $x_i = \epsilon(\widetilde{m}_i)$ for their images in $\mathcal{O}_X$. In other words, $x_i$ are local defining equations for the divisor $D_i$ of $X$ at $x$. Similarly, write $t_i$ for the corresponding images in $S$ near $f(x)$. Then, without loss of generality, we may assume that the map of characteristic monoids has the form 
    \[
    \NN^k=\overline{M}_{S,f(x)} \to \overline{M}_{X,x} =  \NN^{n_1}\oplus \NN^{n_2} \oplus \cdots \oplus \NN^{n_k} \oplus \NN^\ell\, ,
    \]
    with the $j^{th}$ basis element of $\NN^k$
    mapping to the vector $(\underbrace{1,\ldots, 1}_{n_j})$ of the summand
    $\NN^{n_j}$ on the right.
We have equations 
    \[
    t_1 = u_1\prod_{\alpha\in A_1} x_\alpha\,,\,\,\ldots, \,\,t_k=u_k\prod_{\alpha\in A_k}x_\alpha
    \]
 for disjoint sets $A_1, \ldots, A_k\subset \{1, 2, \ldots, n\}$ with $n_1,\ldots,n_k$ elements respectively and units $u_i \in \mathcal{O}_{X,x}$.  By the orientation assumption, the sets $A_1,\ldots,A_k$ are ordered. 
For convenience, we write $$A=A_1\cup\ldots\cup A_k\, .$$ 
 The additional $\ell$ parameters $y_1,\cdots,y_\ell$ 
 have vanishing loci $V(y_i)$ representing horizontal divisors{\footnote{For example, when $X \to S$ is a family of curves, the $V(y_i)$ correspond to markings. In our study
 of the moduli of abelian varieties, $\ell=0$.}}
 over $S$.  
 
 The logarithmic differentials $\Omega_{X,x}^{\log}$ 
    are generated by $\Omega_{X,x}$ and $\frac{dx_{\alpha}}{x_{\alpha}}$, $\alpha \in A$, and $\frac{dy_1}{y_1}, \ldots, \frac{dy_\ell}{y_\ell}$. We have the relations    \[
    \frac{du_i}{u_i}+\sum_{\alpha\in A_i} \frac{dx_{\alpha}}{x_{\alpha}} = \frac{dt_i}{t_i},\, \,\,1\leq i\leq k \,.
    \]
  The quotient $\mathcal{R} = \Omega_{X/S}^{\log}/\Omega_{X/S}$ has a presentation as an $\mathcal{O}_X$-module with generators
 \begin{equation}\label{gener}
    \frac{dx_{\alpha}}{x_{\alpha}}, \frac{dy_1}{y_1}, \ldots,\frac{dy_\ell}{y_\ell}\, ,
 \end{equation} where $\alpha \in A$. 
    The relations are
    \[
    \sum_{\alpha\in A_i} \frac{dx_{\alpha}}{x_{\alpha}} = 0\, ,\,\,\,\, 1\leq i\leq k\,,
    \]
    (since we are working with relative differentials and the $du_i/u_i$ are in $\Omega_{X/S}$) and additionally 
    \begin{align*}
    x_{\alpha}\frac{dx_{\alpha}}{x_{\alpha}} = 0\,, \,\, y_1\frac{dy_1}{y_1}=0\,, \ldots, y_\ell\frac{dy_\ell}{y_\ell}=0\,,
    \end{align*}
where $\alpha \in A$. The irreducible components of the stratification of $M_{X/S}$ at $x$ (with $\text{rk }\ge 1$) are given either by 
\begin{itemize} 
\item $t_i=0,x_{\beta}=0,x_{\gamma}=0$ for triples $(i, \beta, \gamma)$ with $\beta<\gamma$ elements in $A_i$, or by
\item $y_j=0$ for some $1\leq j\leq\ell$. 
\end{itemize} Thus, we find{\footnote{Of course, $t_i\in (x_\beta, x_{\gamma})$, but we have chosen to keep $t_i$ in the notation to emphasize that $\beta,\gamma$ belong to the same part $A_i$.}} 
    \[
    \oplus_H \mathcal{O}_H = \bigoplus_{(i,\beta,\gamma)} \mathcal{O}_X/(t_i,x_{\beta},x_{\gamma}) \oplus \bigoplus_{j} \mathcal{O}_X/(y_j) .
    \] 
    We define a map 
    \[
    \mathcal{R} \to \mathcal{O}_X/(t_i,x_{\beta},x_{\gamma})
    \]
    by sending all the generators in \eqref{gener} to $0$, with the exception of
    \[
    \frac{dx_\beta}{x_\beta} \mapsto 1, \quad \frac{dx_{\gamma}}{x_{\gamma}}\mapsto -1\,.
    \]
    Similarly, we define a map
\[
\mathcal{R} \to \mathcal{O}_X/(y_j)
\]
by sending all generators in \eqref {gener} to $0$, with the exception of
\[
\frac{dy_j}{y_j} \mapsto 1.
\]

We must verify that the map is well-defined. First, for each $(i, \beta,\gamma)$, we see that 
\[
\sum_{\alpha\in A_i} \frac{dx_{\alpha}}{x_{\alpha}} \mapsto 0
\]
since $\frac{dx_{\beta}}{x_{\beta}}$ and $\frac{dx_{\gamma}}{x_{\gamma}}$ map to opposite elements in $ \mathcal{O}_X/(t_i,x_{\beta},x_{\gamma})$, and the other terms map to $0$. The fact that $$x_{\alpha}\frac{dx_{\alpha}}{x_{\alpha}},  \,\,y_1\frac{dy_1}{y_1}, \ldots, y_\ell\frac{dy_\ell}{y_\ell},\, \alpha\in A$$
map to $0$ in $\mathcal{O}_{X}/(t_i,x_{\beta},x_{\gamma})$ and $\mathcal{O}_X/(y_j)$ is immediate from the definitions. Surjectivity of the map to any summand $\mathcal{O}_H$ is also clear, as the generator $1$ is in the image.

We now inspect how our map depended on choices; the only choices involved were the lifts $\widetilde{m}_i$ of $m_i$, and the choice of ordering of the $m_i$. A different choice of $\widetilde{m}_i'$ differs from the original one by a unit, and we have 
\[
\frac{d(ux)}{ux} = u^{-1}du + \frac{dx}{x} .
\]
The term $u^{-1}du$ is an ordinary differential, and thus the residue of the logarithmic form is independent of lift. On the other hand, the map does depend on the ordering of the coordinates. Since we assume that the ordering is global, however, the local maps patch uniquely to all of $X$. 
\end{proof}

\begin{rem}
    In case $S$ is a point, $(X,D)$ is a simplest normal crossings pair, and only the horizontal divisors $H = \mathcal{O}_X/(y_i)$ are present in our analysis. These are precisely the components $D_i$ of the divisor $D$. Our residue map then reduces to the classical residue homomorphism 
    \[
    \begin{tikzcd}
        0 \ar[r] & \Omega_X \ar[r] & \Omega_X^{\log} \ar[r] & \oplus_i \mathcal{O}_{D_i} \ar[r]& 0\,,
    \end{tikzcd}
    \]
    see \cite[Chapter 4, Proposition 1]{Fultontoric}.

    The classical residue map, applied to the base and the source of a morphism of simple normal crossings pairs, is used by Esnault-Viewheg \cite[Claim 4.3]{EV} to construct a relative residue map, similar to the one of  Theorem \ref{thm: residue}.
\end{rem}

\subsubsection{Further remarks} It is possible to give more precise estimates about the image of the map 
\[
\mathcal{R} \to \oplus_H \mathcal{O}_H
\]
of Theorem \ref{thm: residue} after restricting to carefully chosen strata. In the following discussion, we keep the assumptions and notation of Theorem \ref{thm: residue}. 

\begin{definition}
    Let $x$ be a point of $X$. The combinatorial type of $f$ at $x$ is the homomorphism 
     \[
    \NN^k=\overline{M}_{S,f(x)} \to \overline{M}_{X,x} =  \NN^{n_1}\oplus \NN^{n_2} \oplus \cdots \oplus \NN^{n_k} \oplus \NN^\ell\, ,
    \]
   The type of a stratum $Y \subset X$ is the type of $f$ at the generic point of $Y$. 
\end{definition}
The combinatorial type of $f$ at $x$ is determined by natural numbers $k,\ell$ and a partition $A$ of a natural number $n$ into $k$ non-empty parts $A_1,\cdots,A_k$ of cardinality $n_1,\cdots,n_k$ respectively. The components $H$ of Theorem \ref{thm: residue} that contain $x$ then split into two types: 
\begin{itemize}
    \item The components of the singular locus $H_{\beta\gamma} = \Spec_X \mathcal{O}_X/(t_i,x_\beta,x_\gamma)$, for $\beta < \gamma \in A_i$. 
    \item The horizontal divisors $H_j = \Spec_X \mathcal{O}_X/(y_j)$, $1 \le j \le \ell$. 
\end{itemize}


\begin{cor}
    \label{cor: residueimage}
    Let $f: X \to S$ be a log family of simplest normal crossings pairs, with a choice of orientation on $\overline{M}_X$. Let $Y \subset X$ be a stratum of type $(k,\ell,A)$. For each $1 \le i \le k$, let $\beta_i$ be the smallest element of $A_i$, and define   
    \[
    B_i = \{(\beta_i,\gamma): \gamma \in A_i,\, \beta_i \neq \gamma\} \, .
    \]
    Let $B = \cup_{i=1}^k B_i$. Then the composition of the residue map of Theorem \ref{thm: residue} with the projection 
    \[
    \mathcal{R} \to \oplus_{H} \mathcal{O}_H \to \bigoplus_{(\beta,\gamma) \in B} \mathcal{O}_{H_{\beta\gamma}} \oplus \bigoplus_{j=1}^\ell \mathcal{O}_{H_j}
    \]
    is surjective when restricted to $Y$. 
\end{cor}

\begin{proof}
    We inspect the map constructed in Theorem \ref{thm: residue}. The logarithmic differentials 
    \[
    \frac{dy_j}{y_j}
    \]
    map to $1$ in $\mathcal{O}_{H_j}$ and to $0$ in all other copies of $\mathcal{O}_{H_{j'}}, j' \neq j$ and in $\mathcal{O}_{H_\beta\gamma}$. Similarily, for each $(\beta_i,\gamma)$ in $B_i$, the logarithmic differential 
    \[
    \frac{dx_{\beta_i}}{x_{\beta_{i}}} + \sum_{(\beta_i,c) \in B_i, c \neq \gamma} \frac{dx_c}{x_c}
    \]
    maps to $1$ in $\mathcal{O}_{H_{\beta_i\gamma}}$, $0$ in $\mathcal{O}_{H_{\beta\gamma}}$ for $(\beta,\gamma) \in B$ different from $(\beta_i,\gamma)$, and to $0$ in $\mathcal{O}_{H_j}$. 
\end{proof}

Keeping the notations of Theorem \ref{thm: residue} and Corollary \ref{cor: residueimage}, we can then state: 

\begin{cor}
    \label{cor: residueandcodimension} 
    Let $f:X \to S$ be a log family of simplest normal crossings pairs, and let $Y \subset X$ be a stratum of relative codimension $m$. Then there exists a subset $B$ of the components $H$ of Theorem \ref{thm: residue} containing $Y$ which has cardinality $m$ and such that the projection  
    \[
    \mathcal{R} \to \oplus_{H} \mathcal{O}_H \to \oplus_{H \in B} \mathcal{O}_H
    \]
    is surjective when restricted to $Y$. 
\end{cor}

\begin{proof}
    Consider the combinatorial type 
    \[
    \NN^k=\overline{M}_{S,f(x)} \to \overline{M}_{X,x} =  \NN^{n_1}\oplus \NN^{n_2} \oplus \cdots \oplus \NN^{n_k} \oplus \NN^\ell\, ,
    \]
    of the morphism at the generic point of $Y$. In order for the relative codimension of $Y$ to be $m$, we need that 
    \[
    \sum_{i=1}^k (n_i-1)+\ell = m
    \]
    Corollary \ref{cor: residueimage} produces a set $B$ of cardinality precisely $m$, as desired. 
\end{proof}

\subsection{The Hodge bundle} \label{sss3} 
Let $\overline{\mathcal{A}}_g$ be a toroidal compactification of the moduli space of principally polarized abelian varieties. The compactification $\overline{\mathcal{A}}_g$ carries a universal family of semi-abelian schemes 
\[
q: \mathcal{U}_g \to \overline{\mathcal{A}}_g
\]
together with a  zero section $s: \overline{\mathcal{A}}_g \to \mathcal{U}_g$. The Hodge bundle is the rank $g$ vector bundle on $\overline{\mathcal{A}}_g$ defined by  
    \[
    \mathbb{E} = s^*\Omega_q\,,
    \]
with Chern classes $\lambda_i=c_i(\mathbb{E})$.

\begin{definition}
    A {\em compactification} of $q: \mathcal{U}_g \to \overline{\mathcal{A}}_g$ is a diagram 
    \[
    \begin{tikzcd}
        \mathcal{U}_g \ar[r] \ar[rd, swap, "q"] & \mathcal{X}_g \ar[d,"p"] \\ 
        & \overline{\mathcal{A}}_g
    \end{tikzcd}
    \]
    where $p$ is a proper log smooth morphism, $\mathcal{U}_g$ is open and dense in $\mathcal{X}_g$, and  $\mathcal{U}_g$ acts on
    $\mathcal{X}_g$ extending the natural action of $\mathcal{U}_g$ on itself (and commuting with $p$). A compactification $p$ is a {\em compactified universal family} if in addition $p$ is a log family.
\end{definition}


An arbitrary toroidal compactification $\overline{\mathcal{A}}_g$ may not carry a compactified universal family. However, toroidal compactifications $\overline{\mathcal{A}}_g$, with compactified universal families 
\[p: \mathcal{X}_g \to \overline{\mathcal{A}}_g\, \] 
can be constructed, see \cite [Chapter VI, Section 1]{FC}. Compactifications of $q$ correspond to $\GL_g \ltimes N$-admissible decompositions $\widetilde{\Sigma}_g$ of a certain subcone of $\Sym^2_{rc}(\mathbb{R}^g) \times \Hom(N,\mathbb{R})$ for a rank $g$ lattice $N$. The decomposition is required to have the property that every cone in $\widetilde{\Sigma}_g$ maps into a cone of the admissible decomposition $\Sigma_g$ of $\Sym^2_{rc}(\mathbb{R}^g)$ defining $\overline{\mathcal{A}}_g$. A compactification $p$ is a compactified universal family if the map $$\widetilde{\Sigma}_g \to \Sigma_g$$ satisfies the additional hypotheses of Definition \ref{def:logfamily} (the cones of $\widetilde{\Sigma}_g$ map onto cones of $\Sigma_g$, and surjectivity also holds for their integral structure). 

Both notions of compactification are stable under arbitrary base change $\overline{\mathcal{A}}_g' \to \overline{\mathcal{A}}_g$. For a compactification $p: \mathcal{X}_g \to \overline{\mathcal{A}}_g$, an arbitrary log alteration $\mathcal{X}_g' \to \mathcal{X}_g$ of the {\em domain} of $p$ remains a compactification. On the other hand, log alterations of $\mathcal{X}_g$ are \emph{not} compactified universal families, even if the original $p$ is such a family, as the composed map $\mathcal{X}_g' \to \overline{\mathcal{A}}_g$ is rarely a log family (flatness and reducedness of fibers are typically destroyed). Nevertheless, semistable reduction by Theorem \ref{thm:semistablereduction} ensures that there is a log alteration of the {\em map} $\mathcal{X}_g' \to \overline{\mathcal{A}}_g$ which is a compactified universal family. 

The sheaf of relative logarithmic differentials of $q$ and $p$ are fiberwise trivial of rank $g$ \cite[Chapter VI, Theorem 1.1]{FC}. 

When $\overline{\mathcal{A}}_g$ has a compactified universal family $p$, we can use $\Omega_p^{\log}$ to define the Hodge bundle, as 
\[
\mathbb{E} = s^*\Omega_q = s^*\Omega_q^{\log} =  s^*\Omega_{p}^{\log}|_{\mathcal{U}_g} =s^*\Omega_p^{\log}
\]
since the section $s$ factors through $\mathcal{U}_g$ and the map $q$ is strict. 

A second approach to  the Hodge bundle is available. The following result can be found in \cite [Chapter VI, Theorem 1.1]{FC}.

\begin{lem}
\label{lem:pushhodge}
    Suppose $\overline{\mathcal{A}}_g$ carries a compactified universal family $p: \mathcal{X}_g \to \overline{\mathcal{A}}_g$. Then,
    \begin{equation*}
    \Omega_p^{\textup{log}} = p^*\mathbb{E}\, \ \ \ \text{and}\ \ \ 
    \mathbb{E} = p_*\Omega_{p}^{\textup{log}}\, .
    \end{equation*}
\end{lem}

\begin{proof}
Since $\Omega_p^{\textup{log}}$ is fiberwise trivial, cohomology and base change implies that 
    \[
    \Omega_p^{\textup{log}} = p^*p_* \Omega_{p}^{\textup{log}}
    \]
    Since 
    $
    s^*p^* = \textup{id},
    $
    we have 
    \[
    \mathbb E = s^*\Omega_p^{\log} = s^*p^*p_*\Omega_p^{\log} = p_*\Omega_p^{\log}\, ,
    \]
    and the result follows. 

\end{proof}

\begin{lem}
\label{lem:extends}
The classes $\lambda_i$ extend to $\mathsf{logCH}^{*}(\A_g)$.
\end{lem}

\begin{proof} In light of Remark \ref{r8}, we check compatibility of Hodge classes under logarithmic alterations. 
    Suppose $\tau: \overline{\mathcal{A}}_g' \to \overline{\mathcal{A}}_g$ is a logarithmic alteration. Then we have a Cartesian diagram 
    \[
    \begin{tikzcd}
        \mathcal{U}_g' \ar[r,"\rho"] \ar[d,"q'"] & \mathcal{U}_g \ar[d,"q"] \\ 
        \overline{\mathcal{A}}_g' \ar[r,"\tau"] & \overline{\mathcal{A}}_g
    \end{tikzcd}
    \]
    Since $s \circ \tau = \rho \circ s'$ for the respective zero sections, we have 
    \[
    \tau^*\mathbb{E} = (s')^*\rho^* \Omega_q = (s')^*\Omega_{q'} = \mathbb{E}'\, ,
\]
which implies the required compatibility for $\lambda_i$.
\end{proof}

\begin{rem}
    \label{rem:simplestncreduction}

Toroidal compactifications of $\A_g$ with a compactified universal family $$p:\mathcal{X}_g \to \overline{\mathcal{A}}_g$$ form a cofinal system among all toroidal compactifications: given an arbitrary toroidal compactification $\overline{\mathcal{A}}_g'$, we can choose a
toroidal compactification $\overline{\mathcal{A}}_g$ with a compactified universal family, and then any common refinement $\overline{\mathcal{A}}_g''$ of both compactifications carries a compactified universal family. 

\end{rem}

\begin{lem} \label{snclem}
The collection of simplest normal crossings compactifications $\mathcal{A}_g \subset \overline{\mathcal{A}}_g$ that carry a compactified universal family with simplest normal crossings is cofinal among the toroidal compactifications $\overline{\mathcal{A}}_g$.
\end{lem}

\begin{proof}
    Starting with an arbitrary compactified universal family 
    \[
    p:\mathcal{X}_g \to \overline{\mathcal{A}}_g
    \]
    we may apply Corollary \ref{cor:simplestncreduction} to $p$, to obtain the desired family. 
\end{proof}

\subsection{Proof of Theorem \ref{thm:main}.} \label{sss4}

Let $\overline{\A}_g$ be a toroidal compactification of $\A_g$. By Lemma \ref{snclem}, there exists a logarithmic alteration $p:\overline{\A}'_g \to \mathcal{\A}_g$ satisfying the following
conditions:
\begin{enumerate}
    \item [(i)] $\overline{\A}'_g$ admits a compactified universal family $\X_g'\rightarrow \overline{\A}'_g$, 
    \item [(ii)] both $\overline{\A}'_g$ and $\X_g'$ have simplest normal crossings.
\end{enumerate}
By Lemma \ref{lem:extends}, $\lambda_g$ defined on $\overline{\A}_g'$ via its own Hodge bundle agrees with the pullback of $\lambda_g$ from $\overline{\A}_g$. Since $\overline{\A}'_g \to \overline{\A}_g$ is a log alteration, it is proper and surjective; furthermore, $p$ is an isomorphism over $\A_g$, and sends the boundary of $\overline{\A}_g'$ to the boundary of $\overline{\A}_g$. Because proper surjections are Chow envelopes, it therefore suffices to show that $$\lambda_g\big{|}_{\partial \overline{\A}'_g}=0\,. $$ Hence, after replacing $\overline{\A}_g$ by $\overline{\A}'_g$, we may assume that $\overline{\A}_g$ has properties (i) and (ii) above.

Let $T$ be a component of the boundary divisor of $\overline{\A}_g$ and denote by $p_T:\X_{T}\rightarrow T$ the base change to $T$ of the universal family $$p:\mathcal X_g\to \overline{\A}_g\,.$$ 

Since $T$ lies in the boundary of $\overline{\A}_g$, the singular locus of $p_T$ is non-empty. Let $H$ be an irreducible component of the singular locus of $p_T$. Since $p_T$ is logarithmically smooth, $H$ is an irreducible component of the locus where $\textup{rank }\overline{M}_{\mathcal{X}_g/\overline{\mathcal{A}}_g} \ge 1$. Let $$i:H\rightarrow \X_T\, , \quad p_T \circ i:H\rightarrow T$$ be the inclusion and the projection. The map $p_T \circ i$ is proper and surjective because $p$ is a log family, so it suffices to show that $i^*p_T^*\left(\lambda_g\big{|}_{T}\right)=0$. Using Lemma \ref{lem:pushhodge}, we have  $\Omega_p^{\textup{log}}=p^*\mathbb E$. After base change and pullback by $i$, we find
\[
i^*p_T^* \mathbb{E} \big{|}_{T} = i^*\Omega_{p_T}^{\textup{log}}\, .
\]
It remains to check that $c_g(i^*\Omega_{p_T}^{\textup{log}}) = 0$. By Definition \ref{def:residues} and Theorem \ref{thm: residue}, we have surjections
\[
i^*\Omega_{p_T}^{\log} \to  i^*\mathcal{R} \to 0\, , \ \ \ 
i^*\mathcal{R} \rightarrow \mathcal{O}_H\to 0\,.
\]

We therefore have an exact sequence of vector bundles on $H$,
\[
\label{eq: residuerank1quotient}
\begin{tikzcd}
0 \ar[r] & K \ar[r] & i^*\Omega_{p_T}^{\log} \ar[r] & \mathcal{O}_H \ar[r] & 0\,.
\end{tikzcd}
\]
We conclude
$
c_g(i^*\Omega_{p_T}^{\log}) = c_{g-1}(K)c_1(\mathcal{O}_H) = 0\,. 
$
\qed
\vskip.1in
As in Corollary \ref{cor: residueimage}, stronger vanishing results can be obtained if we restrict to more carefully chosen strata of the boundary of $\overline{\A}_g$. 

\begin{thm}
    \label{thm: highervanishing}
    Let $\overline{\A}_g$ be a toroidal compactification of $\A_g$, and let $$\overline{\A}_g^{\ge m}\subset \overline{\A}_g $$ denote the complement of the open locus where the torus rank of the universal semiabelian family $\mathcal{U}_g \to \overline{\A}_g$ is less than $m$. Then, we have 
    \[
    \lambda_{j}|_{\overline{\A}_g^{\ge m}} = 0 
    \]
    for $g-m+1 \le j \le g$.
\end{thm}

\begin{proof}
    Let $\overline{\A}_g' \to \overline{\A}_g$ be a log alteration. Since the locus $\overline{\A}_g^{\ge m}$ pulls back to the corresponding locus in $\overline{\A}_g'$, we reduce, as in the proof of Theorem \ref{thm:main}, to the case where $\overline{\A}_g$ carries a compactified universal family $p: \mathcal{X}_g \to \overline{\A}_g$. Furthermore, as before, we replace $\overline{\A}_g^{\ge m}$ with one of its irreducible components $T$. 

    In the compactified universal family $\mathcal{X}_g|_T$, the torus part of the semi-abelian scheme $\mathcal{U}_g|_T$ is compactified as a union of proper toric varieties. Since every complete toric variety of dimension $d$ has strata of codimension $i$, $0 \le i \le d$, and the torus part of $\mathcal{U}_g$ over $T$ has dimension $\ge m$, the universal family 
    \[
    \mathcal{X}_g\big{|}_T \to T
    \]
    has a stratum $Y$ of relative codimension $\ge m$ over $T$.\footnote{That is, the map $Y \to T$ is flat and the restriction of $Y$ to each fiber has codimension $\ge m$.} Combining the proof of Theorem \ref{thm:main} with Corollary \ref{cor: residueandcodimension}, we find that the pullback of the Hodge bundle $\mathbb{E}_g$ to $Y$ has a rank $m$ trivial quotient. Therefore, the top $m$ Chern classes 
    of $\mathbb{E}_g$
    restricted to $Y$ are $0$. Since $Y \to T$ is proper and surjective, the  top $m$ Chern classes of $\mathbb{E}_g$ restricted to $T$ are $0$.     
\end{proof}

\subsection{Log geometry and \texorpdfstring{$\lambda_g$}{}.}
\label{ppp}

There is a distinguished subalgebra of classes coming from the
boundary in the logarithmic Chow theory defined by the image of the
algebra $\mathsf{PP}$ of $\GL_g$-invariant piecewise polynomials\footnote{By definition, a piecewise polynomial on $\Sym^2_{\mathrm{rc}}(\mathbb{R}^g)$ is an admissible decomposition together with a continuous $\GL_g$-invariant function on the decomposition that is polynomial on each cone.} on $\Sym^2_{\mathrm{rc}}(\mathbb{R}^g)$,
\begin{equation}\label{impp}
\mathsf{PP}\to \mathsf{log CH}^*({\mathcal{A}}_g)\, .
\end{equation}
We refer the reader to \cite {MPS,MRan} for further details regarding the construction of the map \eqref{impp}. Our main conjecture concerning $\lambda_g$ in the logarithmic theory is the
following claim.
\vskip.1in

\noindent {\bf Conjecture D.}
{\it The class
$\lambda_g \in  \mathsf{log CH}^*(\A_g)$
lies in the image \eqref{impp} of  the algebra of piecewise polynomials.}
\vskip.1in

Our motivation for Conjecture D comes from a parallel
study of $\lambda_g$ in the logarithmic Chow theory of the moduli
space of curves $\mathcal{M}_g$. Using Pixton's formula \cite{HMPPS,JPPZ},
the class $\lambda_g$ is proven in \cite{MPS} to lie in the image of the algebra
of piecewise polynomials in $\mathsf{log CH}^*(\mathcal{M}_g)$.

\vspace{8pt}
\noindent {\bf Question E.} {\it Find a lifting 
to $\mathsf{log CH}^*(\mathcal{A}_g)$
of Pixton's formula for 
$\lambda_g \in \mathsf{log CH}^*(\mathcal{M}_g)$
which is compatible with the Torelli map.}
\vspace{8pt}

The definition by van der Geer of the tautological ring is best suited for studying classes on the moduli space of abelian varieties $\A_g$.
There is a larger  tautological ring which takes the boundaries of the various compactifications into account,
$$ \mathsf{log R}^*({\mathcal{A}}_g)
\subset \mathsf{log CH}^*(\mathcal{A}_g)\, ,$$
defined to
be generated by all possible piecewise polynomial and Hodge classes
on all boundary strata of all toroidal compactifications $\A_g \subset \overline{\A}_g$.

The investigation of the structure of  $\mathsf{log R}^*({\mathcal{A}}_g)$ is 
an interesting future direction. For example, pushing forward powers of the polarization, we can define $\kappa$ classes over $\overline{\A}_g$, see 
\cite{MOP} and \cite{A2} for similar constructions in the context of the moduli of K3 surfaces and over KSBA moduli respectively. In the case of abelian varieties, we expect{\footnote{We thank V. Alexeev for a discussion about $\kappa$ classes at the conference 
{\em Higher Dimensional Algebraic Geometry} in La Jolla in January 2024.}} that the $\kappa$ classes lie in
$\mathsf{log R}^*({\mathcal{A}}_g)$. Can a precise formula be found?

\section{Tautological projection of product classes}\label{a1} 
\subsection {Product cycles} We compute here the tautological projections of all product cycles $$\A_{g_1}\times \ldots \times \A_{g_\ell} \rightarrow \A_g$$ for all $g$.  Calculations for product cycles for genus $g\leq 5$ can be found in \cite {GH}.

Fix toroidal compactifications $\overline{\A}_g$ corresponding to an additive collection of fans. The product maps
\[
\prod_{g_1+\dots+g_\ell=g}\A_{g_1}\times\dots\times\A_{g_\ell}\rightarrow \A_g
\]
then extend to maps
\begin{equation} \label{vcc3}
\prod_{g_1+\dots+g_\ell=g}\overline{\A}_{g_1}\times\dots\times\overline{\A}_{g_\ell}\rightarrow \overline{\A}_g\, .
\end{equation}
For example, we could take the perfect cone compactification for every $g$ by \cite[Lemma 2.8]{SB}. 

The Hodge bundle splits canonically over the product \eqref{vcc3}. Indeed, the universal semiabelian variety restricts in the natural fashion over the product, and the splitting of the Hodge bundle then follows by restricting the relative cotangent bundle to the zero section.

\subsection{Lagrangian Grassmannian}\label{secLG}
As remarked in \cite {vdg},
a consequence of Theorem
\ref{vdgthm2} is the
$\mathbb{Q}$-algebra isomorphism
$$\R^*(\overline {\A}_g)\simeq \CH^*(\LG_g)\,,$$ where $\mathsf {LG}_g$ denotes the Lagrangian Grassmannian of  $\mathbb C^{2g}$ with respect to
a symplectic form $\omega$.
An overview of the cohomology of the Lagrangian Grassmannian from the point of Schubert calculus can be found, for instance, in \cite{FP, KT,Pr}.

The spaces $\overline{\A}_g$ and $\LG_g$ are further connected by the Hirzebruch-Mumford proportionality principle \cite{Mum, vdg3}. Let $\mathsf S\to \LG_g$ be the universal rank $g$ subbundle, and let $x_i=c_i(\mathsf S^*).$ Then, 
\begin{equation}\label{vvvggg}
\int_{\overline{\A}_g}\lambda_I=\gamma_g \int_{\LG_g}x_I\, 
\end{equation}
for every $I\subset \{1, 2, \ldots, g\}$. Here, we use the multindex notation $$\lambda_I=\prod_{i\in I}\lambda_i\, , \quad x_I=\prod_{i\in I} x_i\,.$$ The proportionality constant $\gamma_g$ was computed in \cite[page 9]{vdg}:  \begin{equation*}\label{exval}\gamma_g=\int_{\overline{\A}_g} \lambda_1 \ldots\lambda_g=\prod_{i=1}^{g} \frac{|B_{2i}|}{4i}\, .\end{equation*}

For any partition $$g_1+\ldots+g_\ell=g\, ,$$ we can consider the product \begin{equation}\label{lgcycle}\LG_{g_1}\times \ldots \times \LG_{g_\ell}\to \LG_{g}\, .\end{equation} Finding the class of $$\overline{\A}_{g_1}\times \ldots \times \overline{\A}_{g_\ell}\rightarrow \overline{\A}_g$$ is equivalent to finding the class of the product cycle \eqref{lgcycle} in $\CH^*(\LG_g)$ in terms of the Chern classes $x_i=c_i(\mathsf S^*)$ of the dual subbundle. More precisely, if \begin{equation}\label{lgeqn}[\LG_{g_1}\times \ldots \times \LG_{g_\ell}]=\mathsf P(x_1, \ldots, x_g) \in \CH^*(\LG_g)
\,,\end{equation} then we have \begin{equation}\label{ageqn}\mathsf{taut}^{\mathsf{cpt}}\left([\overline{\A}_{g_1}^{\Sigma_{g_1}}\times \ldots \times \overline{\A}_{g_\ell}^{\Sigma_{g_\ell}}]\right)=\frac{\gamma_{g_1}\cdots \gamma_{g_\ell}}{\gamma_g} \cdot \mathsf P(\lambda_1, \ldots, \lambda_g)\in \R^*(\overline {\A}_g^{\Sigma_g})\,.\end{equation} To derive \eqref{ageqn} from \eqref{lgeqn}, we use the Gorenstein property
of $\R^*(\overline {\A}_g^{\Sigma_g})$.
We need only check that polynomials of complementary degrees  in the $\lambda$
classes
pair equally with both sides of \eqref{ageqn}:
\begin{itemize}
\item [(i)] When restricted to the product loci in $\overline{\A}_g^{\Sigma_g}$ and $\LG_g$, both the Hodge bundle $\mathbb E$ and the dual subbundle $\mathsf S^*$ split as direct sums. 
\item [(ii)] By the Hirzebruch-Mumford proportionality principle, integrals in the $\lambda$'s over $\overline{\A}_g^{\Sigma_g}$ can be evaluated in terms of integrals in $x$'s over $\LG_g$. The answers are always proportional \eqref{vvvggg}, with proportionality constant $\gamma_g$. 
\end{itemize}
Combining (i) and (ii), we see that the constant $\gamma_{g_1}\cdots \gamma_{g_\ell}$ arises for all factors on the left hand side, while the constant $\gamma_g$ arises for all terms on the right hand side, showing that \eqref{lgeqn} implies \eqref{ageqn}. 
The purely non-tautological part of the cycle in \eqref{ageqn} plays no role
in the argument.

\subsection{Proof of Theorem \ref{t25}}
For $g_1+\ldots+g_\ell=g$, we must show that the tautological projection of the product locus $\overline{\A}_{g_1}\times \ldots \times \overline{\A}_{g_\ell}$ in $\overline{\A}_g$ is given by the $g\times g$ determinant $$\mathsf{taut}^{\mathsf{cpt}} ([\overline{\A}_{g_1}\times\ldots \times\overline{\A}_{g_\ell}])= \frac{\gamma_{g_1}\ldots\gamma_{g_\ell}}{\gamma_g} \begin{vmatrix}\lambda_{\alpha_1} & \lambda_{\alpha_1+1} & \ldots & \lambda_{\alpha_1+g-1} \\ \lambda_{\alpha_2-1} & \lambda_{\alpha_2} & \ldots & \lambda_{\alpha_2+g-2} \\ \ldots & \ldots & \ldots & \ldots \\ \lambda_{\alpha_g-g+1} & \lambda_{\alpha_g-g+2} & \ldots & \lambda_{\alpha_g}\end{vmatrix} $$ for the vector $$\alpha = (\underbrace{g-g_1, \ldots, g-g_1}_{g_1}, \underbrace{g-g_1-g_2, \ldots, g-g_1-g_2}_{g_2}, \ldots, \underbrace {g-g_1-\ldots-g_{\ell},\ldots, g-g_1-\ldots-g_{\ell}}_{g_{\ell}})\, .$$ 
By the connection between product cycles on $\overline{\A}_g$ and $\LG_g$ proven in
Section \ref{secLG}, it suffices to show that the class of the product $\LG_{g_1}\times \ldots\times \LG_{g_\ell}$ in $\LG_g$ is given by the determinant $$\begin{vmatrix} x_{\alpha_1} & x_{\alpha_1+1} & \ldots & x_{\alpha_1+g-1} \\ x_{\alpha_2-1} & x_{\alpha_2} & \ldots & x_{\alpha_2+g-2} \\ \ldots & \ldots & \ldots & \ldots \\ x_{\alpha_g-g+1} & x_{\alpha_g-g+2} & \ldots & x_{\alpha_g}\end{vmatrix}.$$ We will prove this determinantal formula using the geometry of $\LG_g$. 

Let $V=\mathbb{C}^{2g}$ with symplectic form $\omega$. We consider a splitting 
\begin{equation} \label{gbgb}
(V, \omega)\simeq (V_1, \omega_1)\oplus \ldots \oplus (V_\ell, \omega_\ell)\,,
\end{equation}
where $V_1, \ldots, V_\ell$ are symplectic subspaces of $V$ with $\dim V_i=2g_i$. The splitting \eqref{gbgb} 
defines an embedding $$j:\LG_{g_1}\times \ldots \times \LG_{g_\ell}\to \LG_g\, , \quad (P_1, \ldots, P_\ell)\mapsto P=P_1\oplus \ldots\oplus P_\ell\, .$$ 

Consider the embedding of $\LG_g$ into the usual Grassmannian $\mathsf G=\mathsf G(g, 2g)$: $$\iota: \LG_g\to \mathsf{G}\,.$$ Let $\mathsf S\to \mathsf G$ be the universal subbundle (which   restricts to the universal subbundle $\mathsf{S} \rightarrow \LG_g$ via the embedding $\iota$). Similarly,  let $x_i$ be the Chern classes of $\mathsf{S}^*$ on $\mathsf{G}$ (which restrict to the classes $x_i$
on $\LG_g$). Let $\Sigma$ be the Schubert cycle in the Grassmannian $\mathsf{G}$ 
associated to the partition $\alpha$ with respect to any complete flag $F_1\subset F_2\subset \ldots \subset F_{2g}=\mathbb C^{2g}$ satisfying the property $$F_{2(g_1+\ldots+g_i)}=V_1\oplus \ldots\oplus V_i\, , \quad 1\leq i\leq \ell\, .$$ By definition, $P\in \Sigma$ provided $$\dim (P\cap F_{g+j-\alpha_j})\geq j\, , \quad 1\leq j\leq g\, .$$ For $1\leq i\leq \ell$, let $j=g_1+\ldots+g_i$, so that $\alpha_j=g-(g_1+\ldots+g_i).$ We see that for $P\in \Sigma$ we have 
\begin{equation}\label{schub}\dim \left(P\cap F_{2(g_1+\ldots+g_i)}\right)\geq g_1+\ldots+g_i\, , \quad 1\leq i\leq \ell\, .\end{equation} The converse is also true. While there are additional requirements about dimensions of intersections with other members of the flag, these are automatically fulfilled by elementary linear algebra considerations.

In $\CH^*(\mathsf{G})$, we have the standard expression \cite [Chapter 14]{Fulton}: $$[\Sigma]=\begin{vmatrix}x_{\alpha_1} & x_{\alpha_1+1} & \ldots & x_{\alpha_1+g-1} \\ x_{\alpha_2-1} & x_{\alpha_2} & \ldots & x_{\alpha_2+g-2} \\ \ldots & \ldots & \ldots & \ldots \\ x_{\alpha_g-g+1} & x_{\alpha_g-g+2} & \ldots & x_{\alpha_g}\end{vmatrix} .$$ Moreover, we have $$\codim (\Sigma, \mathsf G)=|\alpha|=\sum_{i=1}^{\ell} (g-g_1-\ldots -g_i) g_i=\sum_{i>j} g_i g_j\, ,$$ which agrees with $$\codim (\LG_1\times \ldots \times \LG_{g_\ell}, \LG_g)=\binom{g+1}{2}-\sum_{i=1}^{\ell} \binom{g_i+1}{2}=\sum_{i>j} g_i g_j\, .$$
The scheme-theoretic claim \begin{equation}\label{di}\LG_{g_1}\times \ldots \times \LG_{g_\ell}=\iota^{-1} \Sigma=\Sigma\cap \LG_g\,,\end{equation} then implies \begin{equation}\label{llgg}[\LG_{g_1}\times \ldots\times \LG_{g_\ell}]=\iota^*[\Sigma]=\begin{vmatrix}x_{\alpha_1} & x_{\alpha_1+1} & \ldots & x_{\alpha_1+g-1} \\ x_{\alpha_2-1} & x_{\alpha_2} & \ldots & x_{\alpha_2+g-2} \\ \ldots & \ldots & \ldots & \ldots \\ x_{\alpha_g-g+1} & x_{\alpha_g-g+2} & \ldots & x_{\alpha_g}\end{vmatrix} ,\end{equation} as required. 

We first establish \eqref{di} set-theoretically. The left to right containment  is clear for split subspaces $P=P_1\oplus \ldots\oplus P_{\ell}$, so we show the converse. Let $P\in \Sigma\cap \LG_g$. For convenience, write $$h_i=g_1+\ldots+g_i\, .$$ We set $P_i=P\cap V_{i}.$ Note that $P\cap F_{2h_i}$ is isotropic in $F_{2h_i}$, hence $\dim (P\cap F_{2h_i})\leq h_i.$ By the Schubert condition \eqref{schub}, we must have \begin{equation}\label{sc}
\dim (P\cap F_{2h_i})= h_i\, .\end{equation} We will prove that $\dim P_i=g_i$ for all $1\leq i\leq \ell.$ 

The case $i=1$ is clear by \eqref{sc} since $V_1=F_{2h_1}$. For the general case, we induct on $i$. We assume that $$\dim (P\cap V_1)=g_1,\,\,\ldots,\,\,\dim (P\cap V_{i})=g_i\,,$$ and show that $$\dim (P\cap V_{i+1})=g_{i+1}\, .$$ To this end, let $Q=P\cap F_{2h_{i+1}}$, so that $$\dim Q=h_{i+1}\, , \quad \dim (Q\cap F_{2h_i})=h_i$$ by \eqref{sc}. Furthermore, $Q$ is isotropic hence Lagrangian in $(F_{2h_{i+1}}, \eta)$ where $\eta$ is the restriction of the symplectic form $\omega$. To show $$\dim (P\cap V_{i+1})=\dim (Q\cap V_{i+1})=g_{i+1}\, ,$$ we compute $$\dim (Q\cap V_{i+1})=\dim Q+\dim V_{i+1} - \dim (Q+V_{i+1})=h_{i+1}+2g_{i+1}-\dim (Q+V_{i+1})\, .$$ It suffices then to show that $\dim (Q+V_{i+1})=h_{i+1}+g_{i+1}$, or equivalently, 
\begin{equation}
    \label{frexx} \dim (Q+V_{i+1})^{\eta}=2h_{i+1}-(h_{i+1}+g_{i+1})=h_i\, . \end{equation} 
    Here, the complement is taken in $F_{2h_{i+1}}$. Since $Q$ is Lagrangian, $Q^{\eta}=Q$. By construction, $V_{i+1}^{\eta}=F_{2h_i}$. We can therefore 
rewrite \eqref{frexx} as 
$$\dim (Q\cap F_{2h_i})=h_i\, ,$$ which is correct by the Schubert condition \eqref{sc}. The inductive step is proven. 

Since $P_1\oplus \ldots\oplus P_\ell\subset P$, equality must hold for dimension reasons. Therefore, $P\in \LG_{g_1}\times \ldots \times \LG_{g_\ell}$, and the proof of the set-theoretic
equality \eqref{di} is complete.

To show \eqref{di} holds scheme-theoretically, it suffices to prove that
the scheme-theoretic intersection $\Sigma\cap \LG_g$ is nonsingular at all  points $P\in \Sigma\cap \LG_g$. Equivalently, we will show \begin{equation}\label{tspace}\dim T_P\, (\Sigma\cap \LG_g)\,=\dim (T_P \,\Sigma\,\cap \,T_P \,\LG_g)\,\leq \dim \LG_{g_1}\times \ldots \LG_{g_\ell} = \sum_{i=1}^{\ell} \binom{g_i+1}{2}\, .\end{equation}

We claim first that all $P\in \Sigma\cap \LG_g$ are nonsingular 
points of the Schubert variety $\Sigma\subset \mathsf{G}$. We use here a result due to \cite {LS}, \cite[Corollary 2.5]{C}: singular points of $\Sigma$ must lie in Schubert varieties for singular partitions associated to $\alpha$, see \cite [Definition 2.1]{C} for the terminology. In our case, nonsingularity at 
$P\in \Sigma\cap \LG_g$ 
is due to the fact that equality holds in \eqref{sc}. Equality \eqref{sc} prevents $P$ from satisfying the Schubert conditions for any of the singular partitions associated to $\alpha$. 

The tangent space of $\Sigma$ at nonsingular points is computed in \cite [Theorem 4.1]{EH}: 
$T_P\Sigma$ is identified with a subspace of the space of linear maps $$\Phi: P \to \mathbb C^{2g}/P$$ 
satisfying the property
$$\Phi:(P\cap F_{2h_i}) \to (P+F_{2h_i})/P\, .$$ 
For tangent space  $T_P\LG_g$, we require $$\Phi:P\to P^*$$ to be symmetric, where we identify $\mathbb C^{2g}/P\simeq P^*$ using the symplectic form. 

Assume $\Phi\in T_P\Sigma \cap T_P\LG_g$.
A straighforward check shows that $$(P+F_{2h_i})/P\simeq F_{2h_i}/(P\cap F_{2h_i})$$ gets identified with $(P\cap F_{2h_i})^*$, so that $$\Phi:(P\cap F_{2h_i})\to (P\cap F_{2h_i})^*\, .$$ We have shown above that $$(P\cap F_{2h_i})=(P\cap V_1) \oplus \ldots \oplus (P\cap V_i)\, .$$ Therefore, $\Phi$ must be symmetric block diagonal with blocks of size $g_1, \ldots, g_\ell$. Equation \eqref{tspace} then follows. \qed

\begin{example}\label{l288}
For $g_1+g_2=g$, the tautological projection of the product locus $\overline{\A}_{g_1}\times \overline{\A}_{g_2}$ is given by the $g_1\times g_1$ determinant \begin{equation}\label{t27}\mathsf{taut}^{\mathsf{cpt}} ([\overline{\A}_{g_1}\times \overline{\A}_{g_2}])= \frac{\gamma_{g_1}\gamma_{g_2}}{\gamma_g} \begin{vmatrix}\lambda_{g_2} & \lambda_{g_2+1} & \ldots & \lambda_{g-1} \\ \lambda_{g_2-1} & \lambda_{g_2} & \ldots & \lambda_{g-2} \\ \ldots & \ldots & \ldots & \ldots \\ \lambda_{g_2-g_1+1} & \lambda_{g_2-g_1+2} & \ldots & \lambda_{g_2}\end{vmatrix} .\end{equation}
The right hand side is the Schur determinant associated to the partition $(\underbrace{g_2, \ldots, g_2}_{g_1})$. The Schur determinant is in general not preserved by exchanging $g_1$ and $g_2$ (which amounts to transposing the partition), but it is so in the presence of Mumford's relation by precisely \cite [Lemma A.9.2]{Fulton}.

\vspace{8pt}
\noindent $\bullet$
In case $g_1=1$, we obtain $$\mathsf{taut}^{\mathsf{cpt}}([\overline{\A}_1\times \overline{\A}_{g-1}])=\frac{g}{6|B_{2g}|}\lambda_{g-1}\, .$$

\noindent $\bullet$ In case
$g=2$,
we obtain $$\mathsf{taut}^{\mathsf{cpt}}([\overline{\A}_2\times \overline{\A}_{g-2}])=\frac{1}{360}\cdot \frac{g(g-1)}{|B_{2g}| |B_{2g-2}|}\cdot (\lambda_{g-2}^2-\lambda_{g-1}\lambda_{g-3})\,.$$ 
\end{example}

\begin{example}\label{l29} For $g_1=\ldots=g_k=1$, $g_{k+1}=g-k$, Theorem \ref{t25} yields\begin{equation}\label{a1a1a1a1}\mathsf{taut}^{\mathsf{cpt}} \left(\left[\underbrace{\overline {\A}_1\times \ldots \times \overline {\A}_1}_{k}\times \overline {\A}_{g-k}\right]\right)=\frac{\gamma_{1}^k\gamma_{g-k}}{\gamma_g} \begin{vmatrix}\lambda_{g-1} & \lambda_{g} & 0 & \ldots & 0 \\ \lambda_{g-3} & \lambda_{g-2} & \lambda_{g-1} & \ldots & 0 \\ \ldots & \ldots & \ldots & \ldots & \ldots \\ \lambda_{g_2-2k+1} & \lambda_{g-2k+2} & \lambda_{g-2k+3}& \ldots & \lambda_{g-k}\end{vmatrix} .\end{equation} For example, we have $$\mathsf{taut}^{\mathsf{cpt}}\left(\left[\overline{\A}_1\times \overline{\A}_1\times \overline{\A}_{g-2}\right]\right)=\frac{1}{36}\cdot \frac{g(g-1)}{|B_{2g}| |B_{2g-2}|}\cdot \left(\lambda_{g-1}\lambda_{g-2}-\lambda_g\lambda_{g-3}\right)\,.$$

\end{example}

\subsection {Proof of Theorem \ref{aagg}}\label{agres} Our goal is
to prove that after restriction to $\A_g$, the tautological projections of the product cycles admit further factorization: 
\begin{equation*}
\mathsf{taut}
{\left([\A_{g_1}\times \cdots \times \A_{g_\ell}]\right)}=\frac{\gamma_{g_1}\ldots\gamma_{g_\ell}}{\gamma_g}\cdot \lambda_{g-1}\cdots \lambda_{g-\ell+1}\cdot \begin{vmatrix} \lambda_{\beta_1} & \lambda_{\beta_1+1} & \ldots & \lambda_{\beta_1+g^*-1} \\ \lambda_{\beta_2-1} & \lambda_{\beta_2} & \ldots & \lambda_{\beta_2+g^*-2} \\ \ldots & \ldots & \ldots & \ldots \\ \lambda_{\beta_{g^*}-g^*+1} &\lambda_{\beta_{g^*}-g^*+2} & \ldots & \lambda_{\beta_{g^*}}\end{vmatrix}\, ,\end{equation*} 
for the vector 
$$\beta=(\underbrace{g^*-g^*_1, \ldots, g^*-g^*_1}_{g^*_1}, \underbrace {g^*-g^*_1-g^*_2, \ldots, g^*-g^*_1-g^*_2}_{g_2^*}, \ldots, \underbrace{g^*-g_1^*-\ldots-g_{\ell}^*}_{g^*_\ell})\, ,$$
where $g^*=g-\ell$ and $g_i^*=g_i-1$. 

The term $\lambda_{g-1}\cdots \lambda_{g-\ell+1}$ is expected to appear in the formula of Theorem \ref{aagg}
by the following reasoning. First,
$$\lambda_{g-m}\cdot [\A_{g_1}\times \ldots \times \A_{g_\ell}]=0\, , \quad 1\leq m\leq \ell-1\, .$$ 
Indeed, the splitting of the Hodge bundle distributes a top Hodge class on at least one of the $\ell$ factors $\A_{g_i}$, yielding the vanishing by Theorem \ref{vdgthm2}(iii). Second, 
we compute the annihilator ideal $$\text{Ann}\,\langle \lambda_{g-1}, \ldots, \lambda_{g-\ell+1}\rangle=\langle \lambda_{g-1}\ldots \lambda_{g-\ell+1}\rangle\, .$$ 
The right to left containment follows from the relations \begin{equation}\label{vv}\lambda_{j}^{2} \lambda_{j+1}\ldots \lambda_{g-1}=0\, , \quad 1\leq j\leq g-1\end{equation} on $\A_g$ noted in \cite [page 4]{vdg}. The left to right inclusion can be justified by expressing an arbitrary element $z$ of the annihilator in terms of the square-free monomial basis in the $\lambda$'s. Using \eqref{vv}, in particular $\lambda_{g-1}^2=0$, it follows that all monomials that appear in $z$ must contain $\lambda_{g-1}.$ If not, $z\cdot \lambda_{g-1}$ would contain nonzero terms in the square-free monomial basis, corresponding to the monomials of $z$ not containing $\lambda_{g-1}$. This contradicts that $z$ is in the annihilator ideal.
Successively, we see that $\lambda_{g-2}, \ldots, \lambda_{g-\ell+1}$ must also appear in each of the monomials of $z$, proving the claim. 

\begin{proof} We only indicate the proof of 
Theorem \ref{aagg}
 when $\ell=2$. The general case is an $\ell$-fold iteration of the same argument. To start, we restrict to $\A_g$ the expression provided by Theorem \ref{t25}, see \eqref{t27}. Then, we must prove $$\begin{vmatrix}\lambda_{g_2} & \lambda_{g_2+1} & \ldots & \lambda_{g-1} \\ \lambda_{g_2-1} & \lambda_{g_2} & \ldots & \lambda_{g-2} \\ \ldots & \ldots & \ldots & \ldots \\ \lambda_{g_2-g_1+1} & \lambda_{g_2-g_1+2} & \ldots & \lambda_{g_2}\end{vmatrix}
=\lambda_{g-1} \cdot \begin{vmatrix} \lambda_{g_2-1} & \lambda_{g_2} & \ldots & \lambda_{g-3} \\ \lambda_{g_2-2} & \lambda_{g_2-1} & \ldots & \lambda_{g-4} \\ \ldots & \ldots & \ldots & \ldots \\ \lambda_{g_2-g_1+1} &\lambda_{g_2-g_1+2} & \ldots & \lambda_{g_2-1}\end{vmatrix}$$ after setting $\lambda_{g}=0.$  The parallel identity for the Lagrangian Grassmannian is equivalent: \begin{equation}\label{toprove}\begin{vmatrix}x_{g_2} & x_{g_2+1} & \ldots & x_{g-1} \\ x_{g_2-1} & x_{g_2} & \ldots & x_{g-2} \\ \ldots & \ldots & \ldots & \ldots \\ x_{g_2-g_1+1} & x_{g_2-g_1+2} & \ldots & x_{g_2}\end{vmatrix}
=x_{g-1} \cdot \begin{vmatrix} x_{g_2-1} & x_{g_2} & \ldots & x_{g-3} \\ x_{g_2-2} & x_{g_2-1} & \ldots & x_{g-4} \\ \ldots & \ldots & \ldots & \ldots \\ x_{g_2-g_1+1} &x_{g_2-g_1+2} & \ldots & x_{g_2-1}\end{vmatrix} \mod x_g\, .\end{equation} The identity does not hold in the absence of the Mumford relations. 

We will derive identity \eqref{toprove}
geometrically
via an excess intersection calculation on $\LG_g$. Fix a symplectic splitting $$V=W_1\oplus U_1\oplus W_2\oplus U_2\, , \quad \dim W_i=2(g_i-1)\, ,\quad \dim U_i=2\, .$$ In addition, fix Lagrangian subspaces $P_1\subset U_1$
and 
$P_2\subset U_2$. Let $$\iota: \LG_{g-1} \to \LG_g\, , \quad P\to P\oplus P_2\, .$$ Here $\LG_{g-1}$ is the Lagrangian Grassmannian of $W_1\oplus U_1\oplus W_2$. For this embedding, we have \begin{equation}\label{repeat}\iota_*\,[\LG_{g-1}]=x_{g}\cap [\LG_g]\, ,\end{equation} as can be seen by a normal bundle calculation. The reader can verify that $$\iota^{-1}(\LG_{g_1}\times \LG_{g_2})=\LG_{g_1}\times \LG_{g_2-1}\, .$$ Here, $\LG_{g_1}$, $\LG_{g_2}$ and $\LG_{g_2-1}$ correspond to $W_1\oplus U_1$, $W_2\oplus U_2$ and $W_2$ respectively. The left hand side has codimension $g_1g_2$ in $\LG_g$, while the right hand side has codimension $g_1(g_2-1)$ in $\LG_{g-1}$. Write $$j:\LG_{g_1}\times \LG_{g_2-1}\to \LG_{g-1}$$ for the natural map determined by the pair $(W_1\oplus U_1, W_2)$. The class $\iota^*(\LG_{g_1}\times \LG_{g_2})$ can be computed via excess intersection. The excess bundle is the dual tautological subbundle $\mathsf S_{g_1}^*$. Therefore, \begin{equation}\label{pull}\iota^* ([\LG_{g_1}\times \LG_{g_2}])=j_{*}((x_{g_1}\times 1)\cap [\LG_{g_1}\times \LG_{g_2-1}])=j_*k_*([\LG_{g_1-1}\times \LG_{g_2-1}])\, ,\end{equation} after using \eqref{repeat} again. The embedding $$k:\LG_{g_1-1}\times \LG_{g_2-1}\to \LG_{g_1}\times \LG_{g_2-1}$$ is defined by taking sum with $P_1$ on the first factor. Consider $$u:\LG_{g_1-1}\times \LG_{g_2-1}\to \LG_{g-2}\, ,\quad v:\LG_{g-2}\to \LG_{g-1}\, ,$$ where the first map is determined by the pair $(W_1, W_2)$ and the second map is determined by taking sum with $P_1$. 
The equality $j\circ k=v\circ u$
follows from the definitions. By \eqref{llgg} in the proof of Theorem \ref{t25}, we find \begin{equation}\label{jj}u_{*}([\LG_{g_1-1}\times \LG_{g_2-1}])=v^*\begin{vmatrix} x_{g_2-1} & x_{g_2} & \ldots & x_{g-3} \\ x_{g_2-2} & x_{g_2-1} & \ldots & x_{g-4} \\ \ldots & \ldots & \ldots & \ldots \\ x_{g_2-g_1+1} &x_{g_2-g_1+2} & \ldots & x_{g_2-1}\end{vmatrix}.\end{equation} Then, using \eqref{pull} and \eqref{jj}, we have \begin{eqnarray*}\iota^*([\LG_{g_1}\times \LG_{g_2}])&=&j_*k_*([\LG_{g_1-1}\times \LG_{g_2-1}])=v_*u_*([\LG_{g_1-1}\times \LG_{g_2-1}])\\&=&v_*v^*\begin{vmatrix} x_{g_2-1} & x_{g_2} & \ldots & x_{g-3} \\ x_{g_2-2} & x_{g_2-1} & \ldots & x_{g-4} \\ \ldots & \ldots & \ldots & \ldots \\ x_{g_2-g_1+1} &x_{g_2-g_1+2} & \ldots & x_{g_2-1}\end{vmatrix}=x_{g-1} \begin{vmatrix} x_{g_2-1} & x_{g_2} & \ldots & x_{g-3} \\ x_{g_2-2} & x_{g_2-1} & \ldots & x_{g-4} \\ \ldots & \ldots & \ldots & \ldots \\ x_{g_2-g_1+1} &x_{g_2-g_1+2} & \ldots & x_{g_2-1}\end{vmatrix}\, ,\end{eqnarray*} which recovers the right hand side of \eqref{toprove}. On the other hand, by \eqref{llgg}, the class on the left hand side equals $$\iota^*\begin{vmatrix}x_{g_2} & x_{g_2+1} & \ldots & x_{g-1} \\ x_{g_2-1} & x_{g_2} & \ldots & x_{g-2} \\ \ldots & \ldots & \ldots & \ldots \\ x_{g_2-g_1+1} & x_{g_2-g_1+2} & \ldots & x_{g_2}\end{vmatrix},$$ while the pullback $\iota^*:\CH^*(\LG_g)\to \CH^*(\LG_{g-1})$ has the effect of setting $x_g=0$. 
\end{proof}

\appendix
\section {Abelian varieties with real multiplication}

Shimura-Hilbert-Blumental varietes parametrize abelian varieties with real multiplication and arise as Noether-Leschetz loci in $\A_g$, see \cite {DebLas}. We propose here a conjecture for the tautological projections of the classes of 
the canonical components of Shimura-Hilbert-Blumenthal varieties in $\A_g$. 

\subsection{Real multiplication} Fix a totally real number field $F$ with $[F:\mathbb Q]=e$.
A polarized abelian variety $X$ admits {\em real multiplication} by $F$ provided
that $$F\subset \text{End}_{\mathbb Q}(X)\,$$ as unital $\mathbb Q$-algebras. In the context of the Noether-Lefschetz loci, the embedding of the totally real field arises from the additional N\'eron-Severi class\footnote{If $L$ is the polarization (on a simple abelian variety), and $M$ is the additional N\'eron-Severi class, the field $F$ is generated by $\phi_{L}^{-1}\phi_M\in \text{End}_\mathbb Q(X)$ where $\phi_L, \phi_M:X\to \widehat X$ are the usual morphisms \cite [Section 2.4]{BL}.}, as explained in \cite{DebLas}. By \cite [Proposition 5.5.7]{BL}, we must have $$g=me$$ for an integer $m$, and the Picard rank of $X$ must be least $e$.

We will use the following notation for objects related to the totally real field $F$:

\vspace{2pt}
\begin{itemize}
\item $\sigma_1, \ldots, \sigma_e:F\to \mathbb R$ are the real embeddings of $F$.
\item $\mathcal O_F$ denotes the ring of integers in $F$.
\item $\mathfrak d^{-1}$ is the codifferent ideal of $F$ given by $$\mathfrak d^{-1}=\{x\in F: \text{Tr}_{F/\mathbb Q}(xy)\in \mathbb Z \quad \text{ for all } y\in \mathcal O_F\}\,.$$ $\mathfrak d^{-1}$ is a fractional ideal whose inverse is the different ideal $\mathfrak d\subset \mathcal O_F$. 
\item The Dedekind zeta function is given by $$\zeta_F(s)=\sum_{\mathfrak a\subset \mathcal O_F} \frac{1}{\left|\mathcal O_F/\mathfrak a\right|^s}, \quad \text{Re }s>1\,,$$ where the sum is taken over ideals $\mathfrak a\subset \mathcal O_F$. The function $\zeta_F$ can be analytically continued to $\mathbb C\smallsetminus \{1\}.$
\end{itemize}
 
\subsection{Shimura-Hilbert-Blumenthal varieties.} 
\subsubsection{Component $\mathcal{A}_F$} If $X$ admits real multiplication by $F$, then
$F \cap \text{End}(X)\subset F$
is an order in $F$. We will consider the canonical component of the 
Shimura-Hilbert-Blumenthal variety of abelian varieties with real multiplication 
by $F$: the component $\mathcal{A}_F$
defined by the condition that the intersection is a maximal order
$$F \cap \text{End}(X)= \mathcal{O}_F\subset F\, .$$
The geometric construction of $\mathcal{A}_F$
is standard \cite {BL, G, S}. 
As in the construction of $\mathcal{A}_g$,
the Siegel upper half
space,
$$\mathfrak H^+_m=\{\tau\in \text{Mat}_{m\times m}(\mathbb C), \quad \tau = \tau^{t}, \quad \text{Im }\tau>0\}\, ,$$ plays
a central role.

\subsubsection{Case $g=e$}

Consider first the case $g=e$ and $m=1$ which corresponds to the classical Hilbert-Blumenthal varieties. For each tuple $$\tau=(\tau_1, \ldots, \tau_g)\in \mathfrak H_1^+\times \ldots\times \mathfrak H_1^+\,,$$ we define $$j_\tau:F\times F\to \mathbb C^g, \quad \begin{pmatrix}x\\ y\end{pmatrix}\mapsto \begin{pmatrix}\sigma_1(x)\tau_1+\sigma_1(y)\\\ldots\\\sigma_g(x)\tau_g+\sigma_g(y)\end{pmatrix}\,.$$ Consider the lattice $$\Gamma'_\tau=j_\tau(\mathcal O_F\times \mathcal O_F)\subset \mathbb C^g\,.$$ The quotient $$X'_{\tau}= \mathbb C^g/\Gamma'_\tau$$ is an abelian variety via the polarization $$H_\tau (z, w) = \sum_{j=1}^{g} \frac{z_j \overline w_j}{\text{Im }\tau_j}\,.$$ 

The abelian variety $(X'_{\tau}, H_\tau)$ may not be principally polarized. To construct a principally
polarized abelian variety, we consider the lattice $$\Gamma_\tau=j_\tau(\mathcal O_F\times \mathfrak d^{-1})\,.$$ The complex torus $$X_{\tau}=\mathbb C^g/\Gamma_\tau$$ is then principally polarized via $H_\tau$. The lattices $\Gamma_\tau$ and $\Gamma'_\tau$ in $\mathbb C^g$ are preserved by the action of $\mathcal O_F$ defined by $$\mathcal O_F\to \text{Mat}_{g\times g} (\mathbb C)\, , \quad x\mapsto \text{diag}(\sigma_1(x)\, , \ldots, \sigma_g(x))\,.$$ 
Therefore, $$\mathcal O_F\hookrightarrow \text{End} (X_{\tau}) \text{ and so }F\subset \text{End} (X_{\tau})_{\mathbb Q}\,.$$ 

There is a moduli space $$\mathcal A_F=\SL_2(\mathcal O_F\oplus \mathfrak d^{-1})\backslash\mathfrak H_1^+\times \ldots\times \mathfrak H_1^+\, .$$ Here, we set $$\SL_2(\mathcal O_F\oplus \mathfrak d^{-1})=\left\{\begin{pmatrix} a & b\\ c & d\end{pmatrix}: a, d\in \mathcal O_F,\quad b\in \mathfrak d, \quad c\in \mathfrak d^{-1}, \quad ad-bc=1\right\}\,.$$ These are matrices that preserve $\mathcal O_F\times \mathfrak d^{-1}$. The action of $\SL_2(\mathcal O_F\oplus \mathfrak d^{-1})$ on the product $\mathfrak H_1^+\times \ldots\times \mathfrak H_1^+$ is given by $$\begin{pmatrix} a & b\\ c & d\end{pmatrix} (\tau_1, \ldots, \tau_g)=\left(\frac{\sigma_1(a)\tau_1+\sigma_1(b)}{\sigma_1(c)\tau_1+\sigma_1(d)}, \ldots, 
\frac{\sigma_g(a)\tau_1+\sigma_g(b)}{\sigma_g(c)\tau_g+\sigma_g(d)}\right)\,.$$ From the above discussion, it follows that there is a 
morphism $\mathcal A_F\to \mathcal A_g$.

The moduli space $\mathcal A_F$ differs slightly from the construction in \cite {BL}, but it agrees with \cite {G, S}. The precise modular interpretation is given in \cite [Theorem 2.17]{G}, for instance. 


\label{geqe}

\subsubsection{Case $g=me$} We consider tuples $$\tau=(\tau_1, \ldots, \tau_e)\in \mathfrak H_m^+\times \ldots \times \mathfrak H_m^+.$$ The  construction of Section \ref{geqe} goes through with minor modifications. We set $$j_\tau:F^m\times F^m\to \mathbb C^g, \quad \begin{pmatrix}x\\ y\end{pmatrix}\mapsto \begin{pmatrix}\tau_1\sigma_1(x)+\sigma_1(y)\\\ldots\\\tau_g\sigma_g(x)+\sigma_g(y)\end{pmatrix}\,,$$ where $x, y\in F^m$, and we can define $\sigma_j:F^m\to \mathbb R^m$ by applying $\sigma_j$ component by component. We consider the lattice $\Gamma_\tau=j_\tau\left(\mathcal O_F^m\times \left(\mathfrak d^{-1}\right)^m\right)\,.$ The abelian variety $X_{\tau}=\mathbb C^g/\Gamma_\tau$ is polarized in the same fashion as above.  
There is a moduli space $$\mathcal A_F=\Sp_{2m}(\mathcal O_F\oplus \mathfrak d^{-1})\backslash \mathfrak H_m^+\times \ldots \times  \mathfrak H_m^+\, $$ of dimension   $$\dim \mathcal A_F= e \cdot \frac{m(m+1)}{2}\,.$$ The group $\Sp_{2m}(\mathcal O_F\oplus \mathfrak d^{-1})$ preserves $\mathcal O_F^m\times \left(\mathfrak d^{-1}\right)^m\subset F^{2m}$ as well as the standard symplectic form $$E(x, y)=\sum_{i=1}^{2m}\text{Tr}_{F/\mathbb Q\,} (x_i y_{i+m}-y_i x_{i+m}), \quad x, y\in F^{2m}.$$

\subsection{Tautological projection} We conjecture a formula for the tautological projection of the cycle obtained by pushing forward the fundamental class $\left[\mathcal A_F\right]$ to $\mathcal A_g$. We provide evidence for the formula via the Hirzebruch-Mumford proportionality principle and Theorem \ref{t25}. 

Our considerations rely on a few assumptions on toroidal compactifications which we now state. Consider the moduli stack $$\mathcal A'_F=\Sp_{2m} (\mathcal O_F) \backslash \mathfrak H_m^+\times \ldots \times \mathfrak H_m^+\,.$$ Let $\overline{\mathcal A}'_F$ be a smooth toroidal compatification with simple normal crossings boundary. We assume 
\begin{itemize} 
\item the Hodge bundle $\mathbb E$ extends to $\overline{\mathcal A}'_F$ and splits as direct sums of bundles of rank $m$ $$\mathbb E=\mathbb E_1\oplus \ldots \oplus \mathbb E_e\, ,$$
\item each one of the bundles $\mathbb E_1, \ldots, \mathbb E_e$ satisfies the Mumford relations,
\item over $\overline{\mathcal A}'_F$, we have 
$$\Omega^{\text{log}}=\Sym^2 \mathbb E_1\oplus \cdots \oplus \Sym^2 \mathbb E_e\,.$$
\end{itemize} 
In fact, we assume a bit more, in particular that there is a compactification $\overline {\mathcal A}_F$ compatible with the choice of $\overline \A'_F$, as well as a morphism $\overline \A_F\to \overline{\A}_g$ extending $\A_F\to \A_g.$ 

The compact dual of $\mathfrak H_m^+\times \ldots \times  \mathfrak H_m^+$ is the product of $e$ Lagrangian Grassmannians $$\LG_m\times \ldots\times \LG_m\,.$$ By the Hirzebruch-Mumford proportionality principle \cite {Mum}, we must have \begin{equation}\label{af}\int_{\overline{\mathcal A}'_F} \mathsf P(\lambda_1, \ldots, \lambda_g) = \gamma'_F \int_{\LG_m\times \ldots\times \LG_m} \mathsf P(x_1, \ldots, x_g)\, \,\end{equation} for some constant $\gamma'_F$. Here, $x_1, \ldots, x_g$ are the Chern classes of the dual bundle $$\mathsf S^*=\mathsf S_1^*\oplus \ldots \oplus \mathsf S_e^*\to \LG_m\times \ldots\times \LG_m\,,$$ where $\mathsf S_i$ is the tautological subbundle on the $i^{\text{th}}$ Lagrangian Grassmannian. 

To find the constant $\gamma'_F$, we use the polynomial $$\mathsf P(\lambda_1, \ldots, \lambda_g)=\lambda_e \lambda_{2e} \cdots \lambda_{me}\,,$$ where $\deg \mathsf P=\dim \mathcal A'_F$. 

For each bundle $\mathcal V$ that satisfies the Mumford relation, we have \begin{equation}\label{sym}\mathsf e(\Sym^2\mathcal V)= 2^g c_1(\mathcal V)\cdots c_g(\mathcal V)\,.\end{equation} The identity can be found in \cite [Section 2]{vdg3}.

We have a logarithmic analogue of Gauss-Bonnet computing the orbifold Euler characteristic\footnote{The result is well-known. See, for example, \cite{Sil}. A proof for Deligne-Mumford stacks appears in \cite[Proposition 2.1]{CMZ}.} $$\chi_{\mathsf{orb}}(\mathcal A'_F) = (-1)^{\dim \mathcal A'_F} \int_{\overline {\mathcal A}'_F} \mathsf e(\Omega^{\text{log}})\, .$$ Thus, using our assumptions, we find \begin{equation}\label{gb}
\chi_{\mathsf{orb}}(\mathcal A'_F)=(-1)^{\dim \mathcal A'_F} \int_{\overline {\mathcal A}'_F}\mathsf e(\Sym^2 \mathbb E_1)\cdots \mathsf e(\Sym^2 \mathbb E_e)\,.\end{equation} Invoking \eqref{sym}, we have  $$\mathsf e(\Sym^2 \mathbb E_1)\cdots \mathsf e(\Sym^2 \mathbb E_e)=\left(2^m \lambda_1^{(1)} \cdots \lambda_m^{(1)}\right)\cdots \left(2^m \lambda_1^{(e)} \cdots \lambda_m^{(e)}\right).$$ Here $\lambda_1^{(i)}, \ldots, \lambda_m^{(i)}$ are the Hodge classes for the summand $\mathbb E_i$. Using the Mumford relations for each of the summands $\mathbb E_i$, the above expression can be rewritten as
$$\left(\lambda_1^{(1)} \cdots \lambda_m^{(1)}\right)\cdots \left(\lambda_1^{(e)} \cdots \lambda_m^{(e)}\right)=\lambda_e \lambda_{2e} \cdots \lambda_{me}\,.$$ Indeed, the last term $\lambda_{me}$ on the right hand side splits as $\lambda_m^{(1)}\cdots \lambda_m^{(e)}$. The next term $\lambda_{(m-1)e}$ is forced to split as $\lambda_{m-1}^{(1)} \cdots \lambda_{m-1}^{(e)}$, since all $\lambda_m^{(i)}$ contributions to this term will cancel when paired with the term $\lambda_{me}= \lambda_m^{(1)}\cdots \lambda_m^{(e)}$, via Mumford's relations, and so on.  
Therefore, 
 \begin{equation}\label{e1}\mathsf e(\Sym^2 \mathbb E_1)\cdots \mathsf e(\Sym^2 \mathbb E_e)=2^g \lambda_e \lambda_{2e} \cdots \lambda_{me}\,.\end{equation} 
The same argument shows  \begin{equation}\label{e2}\mathsf e(\Sym^2 \mathsf S_1^*)\cdots \mathsf e(\Sym^2 \mathsf S_e^*)=2^g x_e x_{2e}\cdots x_{me}\,.\end{equation} Using \eqref{af}, \eqref{gb}, \eqref{e1}, and \eqref{e2}, we find \begin{align*}\chi_{\mathsf{orb}}(\mathcal A'_F)&=(-1)^{\dim \mathcal A'_F} \int_{\overline{\mathcal A}'_F} \mathsf e(\Sym^2 \mathbb E_1)\cdots \mathsf e(\Sym^2 \mathbb E_e)\\&=
(-1)^{\dim \mathcal A'_F} \int_{\overline{\mathcal A}'_F}2^g \lambda_e \lambda_{2e} \cdots \lambda_{me}\\&=(-1)^{\dim \mathcal A'_F}\gamma_F'\int_{\LG_m\times \ldots \times \LG_m}2^g x_e x_{2e}\cdots x_{me}\\&=(-1)^{\dim \mathcal A'_F}\gamma'_F \int_{\LG_m\times \ldots \times \LG_m} \mathsf e(\Sym^2 \mathsf S_1^*)\cdots \mathsf e(\Sym^2 \mathsf S_e^*)\\&=(-1)^{\dim \mathcal A_F'}\gamma'_F \int_{\LG_m\times \ldots \times \LG_m} \mathsf e(\text{Tan}_{\LG_1})\cdots \mathsf e(\text{Tan}_{\LG_m})\\&=(-1)^{\dim \mathcal A_F'}\gamma'_F \cdot \mathsf e(\LG_m)^e=(-1)^{e\cdot \frac{m(m+1)}{2}}\cdot \gamma'_F\cdot (2^m)^e\,.\end{align*}

Using the Harder-Siegel formula \cite [page 493]{H}, \cite {Si}, we have $$\chi_{\mathsf{orb}}(\mathcal A'_F)=\chi_{\mathsf{orb}}(\Sp_{2m}(\mathcal O_F))=\zeta_F(-1)\cdots \zeta_F(1-2m)\,, $$ which yields $$\gamma'_F=(-1)^{e\cdot \frac{m(m+1)}{2}}\cdot \frac{1}{2^g} \cdot \zeta_F(-1)\cdots \zeta_F(1-2m)\,.$$

We can compare integrals of Hodge classes on $\overline{\mathcal A}_F$ and $\overline{\mathcal A}'_F$ using a common finite cover corresponding to the quotient of $\mathfrak H_m^+\times \ldots \times \mathfrak H_m^+$ by $\Sp_{2m}(\mathcal O_F\oplus \mathfrak d^{-1})\cap \Sp_{2m}(\mathcal O_F)\subset \Sp_{2m}(F).$ To this end, we need to be able to make compatible choices of the compactification data. Assuming such choices, we find \begin{equation}\label{aff}\int_{\overline{\mathcal A}_F} \mathsf P(\lambda_1, \ldots, \lambda_g) = \gamma_F \int_{\LG_m\times \ldots\times \LG_m} \mathsf P(x_1, \ldots, x_g)\end{equation} where $$\gamma_F=\gamma'_F\cdot [\Sp_{2m}(\mathcal O_F): \Sp_{2m}(\mathcal O_F\oplus \mathfrak d^{-1})]\,.$$ Here, for two commensurable subgroups $G_1, G_2$ of $\Sp_{2m}(F)$, we write $$[G_1:G_2]=\frac{[G_1:(G_1\cap G_2)]}{[G_2:(G_1\cap G_2)]}\,.$$

Using the definition of the tautological projection and \eqref{aff}, we obtain \begin{align*}\int_{\overline {\mathcal A}_g} \mathsf {taut}^{\mathsf{cpt}}\left([\overline{\mathcal A}_F]\right)\cdot \mathsf P(\lambda_1, \ldots, \lambda_g)&=\int_{\overline{\mathcal A}_F} \mathsf P(\lambda_1, \ldots, \lambda_g)=\gamma_F \int_{\LG_m\times \ldots\times \LG_m} \mathsf P(x_1, \ldots, x_g)\\ &=\gamma_F \int_{\LG_g} [\mathsf \LG_m\times \ldots \times \LG_m] \cdot \mathsf P(x_1, \ldots, x_g)\\&=\gamma_F \int_{\LG_g} \text{Schur}_{m, e} (x_1, \ldots, x_g) \cdot {\mathsf P}(x_1, \ldots, x_g)\\&=\frac{\gamma_F}{\gamma_g} \int_{\overline{\mathcal A}_g} \text{Schur}_{m, e} (\lambda_1, \ldots, \lambda_g) \cdot \mathsf P(\lambda_1, \ldots, \lambda_g)\, ,\end{align*} which implies $$\mathsf {taut}^{\mathsf{cpt}}\left(\left[\overline{\mathcal A}_F\right]\right)=\frac{\gamma_F}{\gamma_g}\cdot \text{Schur} _{m, e}(\lambda_1, \ldots, \lambda_g)\,.$$ On the last line of the above derivation, we have used the Hirzebruch-Mumford principle again. Here, $$\gamma_g = (-1)^{\frac{g(g+1)}{2}}\cdot \frac{1}{2^g} \cdot \zeta(-1)\cdots \zeta(1-2g)$$ is the Hirzebruch-Mumford proportionality constant{\footnote{The expression for $\gamma_g$ given here agrees with the one in equation \eqref{exv} since $\zeta(1-2m)=(-1)^m\frac{|B_{2m}|}{2m}$\,.}}
for $\overline {\mathcal A}_g$ found in \cite [Theorem 1.13]{vdg}. The polynomial $$\text{Schur}_{m, e}(\lambda_1, \ldots, \lambda_g)=\begin{vmatrix}\lambda_{\alpha_1} & \lambda_{\alpha_1+1} & \ldots & \lambda_{\alpha_1+g-1} \\ \lambda_{\alpha_2-1} & \lambda_{\alpha_2} & \ldots & \lambda_{\alpha_2+g-2} \\ \ldots & \ldots & \ldots & \ldots \\ \lambda_{\alpha_g-g+1} & \lambda_{\alpha_g-g+2} & \ldots & \lambda_{\alpha_g}\end{vmatrix}\,$$ corresponds to the partition $$\alpha = (\underbrace{m(e-1), \ldots, m(e-1)}_{m}, \underbrace{m(e-2), \ldots, m(e-2)}_{m}\ldots, \underbrace{m, \ldots, m}_{m})$$ via Theorem \ref{t25}. 
\vskip.1in
\noindent {\bf Conjecture F.} 
{\it The projection of $\left[\overline{\mathcal A}_F\right]$ is computed by  $$\mathsf{taut}^{\mathsf{cpt}} \left(\left[\overline{\mathcal A}_F\right]\right)=\mathsf {const}\cdot \textnormal{Schur}_{m, e} (\lambda_1, \ldots, \lambda_g)\,$$ for the constant $$\mathsf{const}=(-1)^{\frac{g(g-m)}{2}}\cdot \frac{\zeta_F(-1)\cdots \zeta_F(1-2m)}{\zeta(-1)\cdots \zeta(1-2g)} \cdot [\Sp_{2m}(\mathcal O_F): \Sp_{2m}(\mathcal O_F\oplus \mathfrak d^{-1})]\, .$$}
\vskip.1in
The restriction to $\mathcal A_g$ can be slightly simplified using Theorem \ref{aagg}. 
For instance, when $m=1$ and $g=e$, on $\A_g$ we have $$\text{Schur}_{1, g} (\lambda_1, \ldots, \lambda_g)=\lambda_{g-1}\cdots \lambda_1\,,$$ yielding in this case $$\mathsf{taut}\left(\left[{\mathcal A}_F\right]\right)=(-1)^{\frac{g(g-1)}{2}}\frac{\zeta_F(-1)}{\zeta(-1)\cdots \zeta(1-2g)} \cdot [\SL_{2}(\mathcal O_F): \SL_{2}(\mathcal O_F\oplus \mathfrak d^{-1})]\cdot \lambda_{g-1}\cdots \lambda_1\,.$$ 

Aitor Iribar L\'opez has informed us that a proof of Conjecture F together with further results about
the projections of Shimura-Hilbert-Blumenthal varieties will appear in his upcoming paper \cite{IL}. He has also made significant progress on Question B \cite{IL2}.

\end{document}